\documentclass[12pt, a4paper]{amsart}
\usepackage[hmargin=30mm, vmargin=25mm, includefoot, twoside]{geometry}
\usepackage[pdftex]{graphicx}
\usepackage{amssymb,verbatim,amsthm}
\usepackage{latexsym}
\usepackage{mathrsfs}
\usepackage{xspace}
\usepackage{enumerate, paralist}
\usepackage{geometry}

\makeatletter
\def\@tocline#1#2#3#4#5#6#7{\relax
  \ifnum #1>\c@tocdepth 
  \else
    \par \addpenalty\@secpenalty\addvspace{#2}%
    \begingroup \hyphenpenalty\@M
    \@ifempty{#4}{%
      \@tempdima\csname r@tocindent\number#1\endcsname\relax
    }{%
      \@tempdima#4\relax
    }%
    \parindent\z@ \leftskip#3\relax \advance\leftskip\@tempdima\relax
    \rightskip\@pnumwidth plus4em \parfillskip-\@pnumwidth
    #5\leavevmode\hskip-\@tempdima
      \ifcase #1
       \or\or \hskip 1em \or \hskip 2em \else \hskip 3em \fi%
      #6\nobreak\relax
    \dotfill\hbox to\@pnumwidth{\@tocpagenum{#7}}\par
    \nobreak
    \endgroup
  \fi}
\makeatother
\newtheorem{thm}{Theorem}
\newtheorem{prop}[thm]{Proposition}

\newtheorem{cor}[thm]{Corollary}
\newtheorem{conj}[thm]{Conjecture}

\newtheorem{lem}[thm]{Lemma}
\theoremstyle{definition}
\newtheorem{defn}[thm]{Definition}
\newtheorem{remark}[thm]{Remark}
\newtheorem*{gremark}{General remark}
\newtheorem{question}[thm]{Question}
\newtheorem{example}[thm]{Example}
\numberwithin{equation}{section}

\newcommand{\moins}{\,-\!\!\!\!-\,}
\newcommand{\moinss}{-\!\!\!-}
\newcommand{\mk}{\mathfrak}

\newcommand{\Carn}{\textnormal{Carn}}
\newcommand{\Homeo}{\textnormal{Homeo}}
\newcommand{\mono}{\hookrightarrow}
\newcommand{\epi}{\twoheadrightarrow}
\newcommand{\C}{\mathbf{C}}
\newcommand{\R}{\mathbf{R}}
\newcommand{\Z}{\mathbf{Z}}
\newcommand{\GL}{\mathrm{GL}}
\newcommand{\SL}{\mathrm{SL}}
\newcommand{\BS}{\textnormal{BS}}
\newcommand{\FT}{\textnormal{FT}}
\newcommand{\eps}{\varepsilon}
\DeclareMathOperator{\Ker}{Ker}
\DeclareMathOperator{\Hom}{Hom} 
\DeclareMathOperator{\Isom}{Isom}
\newcommand{\bd}{\partial} 
\def\Aut{\mathop{\mathrm{Aut}}\nolimits}

\begin{document}

\title[On the QI classification of LC-groups]{On the quasi-isometric classification of locally compact groups}


%
\author{Yves Cornulier}
\address{CNRS and Laboratoire de Math\'ematiques\\
B\^atiment 425, Universit\'e Paris-Sud 11\\
91405 Orsay\\France}
\email{yves.cornulier@math.u-psud.fr}

\date{November 5, 2016 (later renumbering of sections/results to fit with published version)}
%








\begin{abstract}
This (quasi-)survey addresses the quasi-isometry classification of locally compact groups, with an emphasis on amenable hyperbolic locally compact groups. This encompasses the problem of quasi-isometry classification of homogeneous negatively curved manifolds. A main conjecture provides a general description; an extended discussion reduces this conjecture to more specific statements.

In the course of the paper, we provide statements of quasi-isometric rigidity for general symmetric spaces of noncompact type and also discuss accessibility issues in the realm of compactly generated locally compact groups.
\end{abstract}
\maketitle





\renewcommand{\thesection}{19.\arabic{section}}

\section{Introduction}

\addtocontents{toc}{\protect\setcounter{tocdepth}{1}}

\subsection{Locally compact groups as geometric objects}

It has long been well understood in harmonic analysis (notably in the study of unitary representations) that locally compact groups are the natural objects unifying the setting of connected Lie groups and discrete groups. In the context of geometric group theory, this is still far from universal. For a long period, notably in the 70s, this unifying point of view was used essentially by Herbert Abels, and, more occasionally, some other people including Behr, Guivarc'h, Houghton. The considerable influence of Gromov's work paradoxically favored the bipolar point of view discrete vs continuous, although Gromov's ideas were applicable to the setting of locally compact groups and were sometimes stated (especially in \cite{Gro87}) in an even greater generality.

If a locally compact group is generated by a compact subset $S$, it can be endowed with the word length with respect to $S$, and with the corresponding left-invariant distance. While this distance depends on the choice of $S$, the metric space $(G,d_S)$ ---~or the 1-skeleton of the corresponding Cayley graph~--- is uniquely determined by $G$ up to quasi-isometry, in the sense that if $T$ is another compact generating subset, the identity map $(G,d_S)\to (G,d_T)$ is a quasi-isometry. 


We use the usual notion of Gromov-hyperbolicity \cite{Gro87} for geodesic metric spaces, which we simply call ``hyperbolic"; this is a quasi-isometry invariant. A locally compact group is called {\em hyperbolic} if it is compactly generated and if its Cayley graph with respect to some/any compact generating subset is hyperbolic.

This paper is mainly concerned with the quasi-isometric classification and rigidity of amenable hyperbolic locally compact groups among compactly generated locally compact groups. Within discrete groups, this problem is not deep: the answer is that it falls into two classes, finite groups and infinite virtually cyclic groups, both of which are closed under quasi-isometry among finitely generated groups. On the other hand, in the locally compact setting it is still an open problem. In order to tackle it, we need a significant amount of nontrivial preliminaries. Part of this paper (esp.\ Sections \ref{s_ss} and \ref{qitree}) appears as a kind of survey of important necessary results about groups quasi-isometric to symmetric spaces of noncompact type and trees, unjustly not previously stated in the literature but whose proofs gather various ingredients from existing work along with minor additional features. Although only the case of rank 1 and trees is needed for the application to hyperbolic groups, we state the theorems in a greater generality.


%

\subsection{From negatively curved Lie groups to focal hyperbolic groups}

In 1974, in answer to a question of Milnor, Heintze \cite{Hein} characterized the connected Lie groups of dimension at least 2 admitting a left-invariant Riemannian metric of negative curvature as those of the form $G=N\rtimes\R$ where $\R$ acts on $N$ as a one-parameter group of contractions; the group $N$ is necessarily a simply connected nilpotent Lie group; such a group $G$ is called a {\em Heintze group}. He also showed that any negatively curved connected Riemannian manifold with a transitive isometry group admits a simply transitive isometry group; the latter is necessarily a Heintze group (if the dimension is at least 2). The action of a Heintze group $G=N\rtimes\R$ on the sphere at infinity $\partial G$ has exactly two orbits: a certain distinguished point $\omega$ and the complement $\partial G\smallsetminus\{\omega\}$, on which the action of $N$ is simply transitive. This sphere admits a visual metric, which depends on several choices and is not canonical; however its quasi-symmetric type is well-defined and is a functorial quasi-isometry invariant of $G$. The study of quasi-symmetric transformations of this sphere was used by Tukia, Pansu and R. Chow \cite{tukia86,pansu89m,chow} to prove the quasi-isometric rigidity of the rank one symmetric spaces of dimension at least 3. Pansu also initiated such a study for other Heintze groups \cite{pansu89m,pansu89d}.

On the other hand, hyperbolic groups were introduced by Gromov \cite{Gro87} in 1987. The setting was very general, but for many reasons (mainly unrelated to the quasi-isometric classification), their subsequent study was especially focused on discrete groups with a word metric. In particular, a common belief was that amenability is essentially incompatible with hyperbolicity. This is not true in the locally compact setting, since there is a wide variety of amenable hyperbolic locally compact groups, whose quasi-isometry classification is open at the moment. One purpose of this note is to describe the state of the art as regards this problem. 

In 2005, I asked Pierre Pansu whether there was a known characterization of (Gromov-) hyperbolic connected Lie groups, and he replied me that the answer could be obtained from his $L^p$-cohomology computations \cite{Pan07} (which were extracted from an unpublished manuscript going back to 1995) combined with a vanishing result later obtained by Tessera \cite{Te09}; an algebraic characterization of hyperbolic groups among connected Lie groups, based on this approach, was finally given in \cite{CoTe}, namely such groups are either compact, Heintze-by-compact, or compact-by-(simple of rank one). This approach consisted in proving, using structural results of connected Lie groups, that any connected Lie group not of this form cannot be hyperbolic, by showing that its $L^p$-cohomology in degree~1 vanishes for all~$p$.

In \cite{CCMT}, using a more global approach, namely by studying the class of {\em focal} hyperbolic groups, this was extended to a general characterization of all amenable hyperbolic locally compact groups, and more generally of all hyperbolic locally compact groups admitting a cocompact closed amenable subgroup (every connected locally compact group, hyperbolic or not, admits such a subgroup). 
\begin{thm}[\cite{CCMT}]\label{ccmti}
Let $G$ be a non-elementary hyperbolic compactly generated locally compact group with a cocompact closed amenable subgroup. Then exactly one of the following holds
\begin{enumerate}[(a)]
\item\label{c3} (focal case) $G$ is amenable and non-unimodular. Then $G$ is isomorphic to a semidirect product $N\rtimes\Z$ or $N\rtimes\R$, where the noncompact subgroup $N$ is compacted by the action of positive elements $t$ of $\Z$ or $\R$, in the sense that there exists a compact subset $K$ of $N$ such that $tKt^{-1}\subset K$ and $\bigcup_{n\ge 1}t^{-n}Kt^n=N$.
\item\label{c4} $G$ is non-amenable. Then $G$ admits a continuous, proper isometric, boundary-transitive (and hence cocompact) action on a rank~1 symmetric space of noncompact type, or a finite valency tree with no valency 1 vertex and not reduced to a line (necessarily biregular).
\end{enumerate}
\end{thm}


Groups in (\ref{c3}) are precisely the amenable non-elementary hyperbolic locally compact groups, and are called {\em focal hyperbolic} locally compact groups. 

We assumed for simplicity that $G$ is non-elementary in the sense that its boundary has at least 3 points (and is indeed uncountable), ruling out compact groups and 2-ended locally compact groups (which are described in Corollary \ref{2ended}). It should be noted that any group as in (\ref{c4}) admits a closed cocompact subgroup of the form in (\ref{c3}), but conversely most groups in (\ref{c3}) are not obtained this way, and are actually generally not quasi-isometric to any group as in (\ref{c4}), see the discussion in Section \ref{qia}. This is actually a source of difficulty in the quasi-isometry classification: namely those groups in (\ref{c3}) that embed cocompactly in a non-amenable group bear ``hidden symmetries"; see Conjecture \ref{nnprr} and the subsequent discussion.




For the discussion below, it is natural to split the class of focal hyperbolic locally compact groups $G$ into three subclasses (see \S\ref{tyh} for more details):

\begin{itemize}
\item $G$ is of {\em connected type} if its boundary is homeomorphic to a positive-dimensional sphere, or equivalently if it admits a continuous proper cocompact isometric action on a complete negatively curved Riemannian manifold of dimension $\ge 2$;
\item $G$ is of {\em totally disconnected type} if its identity component is compact, or equivalently if its boundary is a Cantor space, or equivalently if it admits a continuous proper cocompact isometric action on a regular tree of finite valency;
\item $G$ is of {\em mixed type} otherwise; then its boundary is connected but not locally connected.
\end{itemize}

\subsection{Synopsis}

We define the {\em commability} equivalence between locally compact groups as the equivalence relation generated by the requirement that any two locally compact groups $G_1,G_2$ with a continuous proper homomorphism $G_1\to G_2$ with cocompact image are commable. Thus $G_1$ and $G_2$ are commable if and only if there exists an integer $k$, a family of locally compact groups $G_1=H_0,$ $H_1,\dots,H_k=G_2$, and continuous proper homomorphisms $f_i$ with cocompact image, either from $H_i$ to $H_{i+1}$ or from $H_{i+1}$ to $H_i$. 
Obviously, commable compactly generated locally compact groups are quasi-isometric (the converse is not true: after I asked about a counterexample, Carette and Tessera checked that some free products of suitable lattices in Lie groups with $\Z$ are indeed quasi-isometric but not commable, see \S\ref{noncommable}).

A general study of commability is not an easy matter, as it is difficult in general to describe in a satisfactory way, given a locally compact group $G$ (e.g.\ a Baumslag-Solitar group as given in Remark \ref{bs}), those locally groups $H$ in which $G$ embeds as a cocompact subgroup. A study of commability in the realm of focal hyperbolic locally compact groups is carried out in Section \ref{sec:com}, with a comprehensive description (except in the totally disconnected case). Actually, the Mostow rigidity theorem \cite{Mos73} was maybe the first time that quasi-isometries were used, to solve a commability problem.

A comprehensive study of commability in the realm of focal hyperbolic locally compact groups is carried out in Section \ref{sec:com}. 

Section \ref{sec:conclu} addresses the quasi-isometry classification of focal hyperbolic groups, discussing, using the results of all the previous sections, around the following:


\begin{conj}{mainc}[Main conjecture]
Two hyperbolic locally compact groups $G,H$ with $G$ focal are quasi-isometric if and only if they are commable.
\end{conj}


This conjecture can be split between the {\em internal} case ($H$ focal) and the {\em external} case ($H$ non-focal), providing more explicit reformulations, which may seem unrelated at first sight. Notably, the external part of Conjecture \ref{mainc} has the following equivalent restatement:

\begin{conj}{anaqi}[slightly restated]
A focal hyperbolic locally compact group is quasi-isometric to a non-focal hyperbolic locally compact group if and only if it is quasi-isometric to a rank~1 symmetric space of noncompact type or a 3-regular tree.
\end{conj}

Because of the very special role played by rank~1 symmetric spaces and trees, and because the extensive literature about them is not formulated in the context of locally compact groups, the important results concerning their quasi-isometric rigidity are surveyed (and slightly extended) in Sections \ref{s_ss} and \ref{qitree}.


In the external part, the totally disconnected type of Conjecture \ref{mainc} is a bit at odds with the other type because there is a complete understanding of the quasi-isometric classification in this case (Section \ref{qitree}), so it is rather a question about commability itself. See \S\ref{sspecial} (esp.\ Question \ref{sqs}). After being asked in a first version of this survey, it has been solved by M.~Carette \cite{Ca}. 

The internal part of Conjecture \ref{mainc} ($H$ focal) can be split into three cases, according to the type of the focal group $G$ (see \S\ref{fh}). Here the totally disconnected case is an easy theorem rather than a conjecture, since all focal hyperbolic locally compact groups of totally disconnected are in the same commability class (Proposition \ref{tdck}). The remaining cases are the mixed type and the connected type. 




Recall that a {\em purely real Heintze group} is a Heintze group $N\rtimes\R$ as above, for which the action of $\R$ on the Lie algebra of $N$ has only real eigenvalues. Heintze groups are focal of connected type, and actually every focal group of connected type is commable to a purely real Heintze group, unique up to isomorphism (see \S\ref{sct}). This gives a reduction of the internal connected type case of Conjecture \ref{mainc} to the following simpler statement:


\begin{conj}
Any two purely real Heintze groups are quasi-isometric if and only if they are isomorphic.
\end{conj}

The mixed type case can also be reduced to a similar statement (Conjecture \ref{int_mi}), by finding, in each commability class of focal hyperbolic locally compact group of mixed type, a given group $H[\varpi,q]$, depending on three independent ``parameters": a purely real Heintze group $H$, a non-power integer $q$, and a positive real number $\varpi$. 

It turns out (see Theorem \ref{th_mi} extracted from \cite{Co13}) that the quasi-isometry class of $H[\varpi,q]$
\begin{itemize}
\item determines the quasi-isometry class of $H$ (by a simple argument based on the boundary);
\item determines the real number $\varpi$, using a computation of $L^p$-cohomology in degree one by the author and Tessera \cite{CoTe}, strongly inspired by Pansu \cite{Pan07};
\item and finally also determines the non-power integer $q$, by a recent result of Dymarz about the large-scale geometry of focal groups of mixed type, 
relying on the study of the fine metrical structure of their boundary.
\end{itemize}

\begin{gremark}This paper may seem to reduce, for expository matters, the study of the quasi-isometry equivalence relation between hyperbolic groups to the determination of quasi-isometry classes. It should by no means be the only point of view; the study of quasi-isometry invariants, such as various kinds of dimensions and more refined ones, allows to shed light on the fine geometric structure of many groups. For this reason, the large-scale study of real Heintze groups should not be reduced to aiming at proving Conjecture \ref{int_con} (which reduces the QI-classification to a classification up to isomorphism, which is, in a certain sense, a wild problem), and even a proof of the latter would not supersede the relevance of the study of these invariants.
\end{gremark}

Here is an outline of the sequel.

\begin{itemize}
\item Section \ref{s_p} contains some preliminary material, notably relying on \cite{CCMT}.
\item In Section \ref{s_ss}, we give the quasi-isometric rigidity statements for symmetric spaces of noncompact type in the locally compact setting. These results are especially due to Kleiner-Leeb, Tukia, Pansu, R.~Chow, Casson-Jungreis and Gabai.
\item In Section \ref{qitree}, we give the quasi-isometric rigidity statements for trees in the locally compact setting, emphasizing on the notion of accessibility (in its group-theoretic and its graph-theoretic versions). These results are notably due to Stallings, Dunwoody, Abels, Thomassen-Woess, Mosher-Sageev-Whyte, and Kr\"on-M\"oller.
\item Section \ref{sec:com} provides a detailed description of  commability classes between focal groups, and between focal and non-focal groups;
\item The core of this paper is Section \ref{sec:conclu}. It contains a discussion about the main conjecture, the link with its specifications, and surveys the main known cases; the first of which being due to P.~Pansu, while recent progress have notably been made by X.~Xie and T.~ Dymarz.
\end{itemize}


\noindent {\bf Acknowledgements.} I thank Tullia Dymarz, Pierre Pansu, and Romain Tessera for useful discussions, comments, and corrections.

\setcounter{tocdepth}{2}
\tableofcontents


\section{Preliminaries}\label{s_p}
\addtocontents{toc}{\protect\setcounter{tocdepth}{2}}
We freely use the shorthand LC-group for locally compact group, and CGLC-group for compactly generated LC-group.

\subsection{Quasi-isometries}\label{s_qi}\label{s_qii}
\begin{center}{\bf The large-scale language}\end{center}\vspace{-0.1cm}

Recall that a map $f:X\to Y$ between metric spaces is a {\em large-scale Lipschitz map} if there exist $\mu>0$ and $\alpha\in\R$ such that 
\[d(f(x_1),f(x_2))\le \mu d(x_1,x_2)+\alpha\quad\forall x_1,x_2\in X;\] we then say $f$ is $(\mu,\alpha)$-Lipschitz. Maps $f:X\to Y$ are at bounded distance, denoted $f\sim f'$ if $\sup_{x\in X}d(f(x),f'(x))<\infty$; if this supremum is bounded by $\alpha$ we write $f\stackrel{\alpha}\sim f'$.

A quasi-isometry $f:X\to Y$ is a large-scale Lipschitz map such that there exists a large-scale Lipschitz map $f':Y\to X$ such that $f\circ f'\sim\textnormal{id}_Y$ and $f'\circ f\sim\textnormal{id}_X$; the map $f'$ is called an {\em inverse} quasi-isometry to $f$; it is unique modulo~$\sim$.

A large-scale Lipschitz map $X\to Y$ is {\em coarsely proper} if, using the convention $\inf\emptyset=+\infty$, the function $F(r)=\inf\{d(f(x_1),f(x_2)):\;d(x_1,x_2)\ge r\}$ satisfies $\lim_{+\infty}F=+\infty$. 

Every CGLC-group $G$ can be endowed with the left-invariant distance defined by the word length with respect to a compact generating subset. Given any two such distances, the identity map is a quasi-isometry, and therefore the notions of large-scale Lipschitz map, quasi-isometry, etc.\ from or into $G$ are independent of the choice of a compact generating set. 

\begin{center}{\bf Hyperbolicity}\end{center}\vspace{-0.1cm}

A geodesic metric space is {\em hyperbolic} if there exists $\delta\ge 0$ such that for every triple of geodesic segments $[ab],[bc],[ac]$ and $x\in [bc]$ we have $d(x,[ab]\cup [ac])\le\delta$. To be hyperbolic is a quasi-isometry invariant among geodesic metric spaces. Thus a locally compact group is called {\em hyperbolic} if it is compactly generated and the 1-skeleton of its Cayley graph with respect to some/any compact generating subset is hyperbolic. By \cite[\S 2]{CCMT}, this holds if and only if it admits a continuous proper cocompact isometric action on a proper geodesic hyperbolic metric space.

\begin{center}{\bf Metric amenability}\end{center}\vspace{-0.1cm}
A locally compact group with left Haar measure $\lambda$ is {\em amenable} (resp.\ {\em metrically amenable}) if for every compact subset $S$ and $\eps>0$ there exists a compact subset $F$ of positive Haar measure such that $\lambda(SF\smallsetminus F)/\lambda(F)\le\eps$, resp.\ $\lambda(FS\smallsetminus F)/\lambda(F)\le\eps$.

Note that for a left-invariant distance, $FS$ is the ``1-thickening" of $F$ and justifies the adjective ``metric". The following lemma is \cite[Theorem 2]{Tes}, see also \cite[\S 4.C]{CH}.

\begin{lem}\label{amqi}
We have
\begin{enumerate}[(a)]
\item A locally compact group is metrically amenable if and only if it is both amenable and unimodular;
\item to be metrically amenable is a quasi-isometry invariant among CGLC-groups.\qed
\end{enumerate}
\end{lem}

Let us emphasize that being amenable is not a quasi-isometry invariant, in view of the cocompact inclusion $\R\rtimes\R\subset\SL_2(\R)$. The problem asking which amenable CGLC-groups are quasi-isometric to non-amenable ones is a very challenging one, it will be addressed in the context of hyperbolic LC-groups in \S\ref{qia}.

\subsection{Cayley-Abels graph}
For a compactly generated locally compact group, the Cayley graph with respect to some compact generating subset is often convenient, but has the drawback, when $G$ is not discrete, to have infinite valency and in addition the action of $G$ on its Cayley graph is not continuous.

\begin{defn}
A {\em Cayley-Abels graph} for $G$ is a continuous, proper and cocompact action of $G$ on a nonempty finite valency connected graph.
\end{defn}

Of course, if $G$ admits a Cayley-Abels graph, then it admits an open compact subgroup, namely the stabilizer of some vertex. The converse is the following elementary fact due to Abels \cite[Beispiel 5.2]{Abe} (see \cite[\S 11.3]{Mo}).

\begin{prop}[Abels]\label{tdg}
Let $G$ be a locally compact group with a compact open subgroup (i.e., $G$ is compact-by-(totally disconnected)). Then $G$ admits a Cayley-Abels graph, which can be chosen to be vertex-transitive.\qed
\end{prop}


\subsection{Types of hyperbolic groups}\label{tyh}

Gromov \cite[\S 3.1]{Gro87} splits isometric group actions on hyperbolic spaces into 5 types: bounded, horocyclic, lineal, focal, and general type, see \cite[\S 3]{CCMT}, from which we borrow the terminology. When specifying this to the action of a CGLC-group $G$ on itself (or any continuous proper cocompact isometric action of $G$) we get four out of these five types, the first 2 of which are called elementary and the last 2 are called non-elementary:
\begin{itemize}
\item $\partial G$ is empty, $G$ is compact;
\item $\partial G$ has 2 elements, $G$ admits $\Z$ as a cocompact lattice (see Corollary \ref{2ended} for more characterizations)
\item $\partial G$ is uncountable and has a $G$-fixed point: $G$ is called a focal hyperbolic group;
\item $\partial G$ is uncountable and the $G$-action is minimal: $G$ is called a hyperbolic group of general type.
\end{itemize}

Among hyperbolic LC-groups, focal groups can be characterized as those that are amenable and non-unimodular, and general type groups can be characterized as those that are not amenable. Most hyperbolic LC-groups of general type (e.g., discrete ones) do not admit cocompact amenable subgroups, the exceptions being listed in Theorem \ref{ccmti}.


Focal hyperbolic groups soon disappeared from the literature after \cite{Gro87}, because the focus was made on proper actions of discrete groups, for which the applications were the most striking\footnote{There was a semantic shift in the meaning of ``elementary", when the terminology from post-Gromov papers, which was only fit for proper actions of discrete groups, was borrowed instead of referring to the general setting duly considered by Gromov.}. Except in the connected or totally disconnected case (and with another point of view), they reappear in \cite{CoTe}, before they were given a structural characterization in \cite{CCMT}.



\subsection{Boundary}

Let $X$ be a proper geodesic hyperbolic space and $\partial X$ its boundary. Then $\overline{X}=X\cup\partial X$ has natural compact topology, for which $X$ is open and dense.

If $X$ and $Y$ are proper geodesic hyperbolic spaces, every quasi-isometric embedding $f:X\to Y$ has a unique extension $\hat{f}:\overline{X}\to\overline{Y}$ that is continuous on $\partial X$. This extension $\hat{f}$ maps $\partial X$ into $\partial Y$ and is functorial in $f$. Let $\bar{f}:\partial X\to\partial Y$ denote the restriction of $\hat{f}$.
Then $\bar{f}=\bar{g}$ whenever $f$ and $g$ are at bounded distance; in particular, every for every quasi-isometry, $\bar{f}$ is a homeomorphism $\partial X\to\partial Y$ whose inverse is $\bar{g}$, where $g$ is any inverse quasi-isometry for $f$.

The boundary carries a so-called visual metric. For such metrics, the homeomorphic embedding $\bar{f}$ above is a {\em quasi-symmetric} embedding in the sense that there exists an increasing function $F:[0,\infty\mathclose[\to [0,\infty\mathclose[$ such that
\[\frac{d(f(x),f(y))}{d(f(y),f(z))}\le F\left(\frac{d(x,y)}{d(y,z)}\right),\quad\forall x\neq y\neq z\in X.\]

\subsection{Focal hyperbolic groups}\label{fh}\label{fhco}\label{fhtd}\label{fhmi}

We say that an automorphism $\alpha$ of a locally compact group $N$ is {\em compacting} (or is a {\em compaction}) if there exists a compact subset $K$ of $N$ such that $\alpha(K)\subset K$ and $\bigcup_{n\ge 0}\alpha^{-n}(K)=N$; we say that an action of $\Z$ or $\R$ by automorphisms $(\alpha^n)_{n\in\Z}$ or $(\alpha^t)_{t\in\R}$ on $N$ is {\em compacting} if $\alpha^1$ is compacting.

\begin{center}{\bf Focal hyperbolic groups of connected type}\end{center}\vspace{-0.1cm}

Recall that Lie groups are not assumed connected and thus include discrete groups; a locally compact group is by definition compact-by-Lie if it has a compact normal subgroup so that the quotient is Lie. This is equivalent to (connected-by-compact)-by-discrete.

We say that a focal hyperbolic group is of {\em connected type} if it is compact-by-Lie. The following proposition is contained in \cite[Theorem 7.3]{CCMT}.

\begin{prop}\label{foco}Let $G$ be a focal hyperbolic LC-group. Equivalences:
\begin{itemize}
\item $G$ is of connected type;
\item the kernel of its modular function is connected-by-compact;
\item it admits a continuous, proper cocompact isometric action on a homogeneous simply connected negatively curved manifold of dimension $\ge 2$.
\item $G$ has a maximal compact normal subgroup $W$ and $G/W$ is isomorphic to a semidirect product $N\rtimes\Z$ or $N\rtimes\R$, where $N$ is a virtually connected Lie group and the action of $\Z$ or $\R$ is compacting.\qed
\end{itemize}
\end{prop}


\begin{center}{\bf Focal hyperbolic groups of totally disconnected type}\end{center}\vspace{-0.1cm}

Similarly we say that a focal hyperbolic group is of {\em totally disconnected type} if its identity component is compact. From \cite[Theorem 7.3]{CCMT} we can also extract the following proposition.

\begin{prop}\label{fotd}Let $G$ be a focal hyperbolic LC-group. Equivalences:
\begin{itemize}
\item $G$ is of totally disconnected type;
\item the kernel of its modular function has a compact identity component;
\item it admits a continuous, proper cocompact isometric action on a regular tree of finite valency greater than~2.
\item it is isomorphic to a strictly ascending HNN-extension over a compact group endowed with an injective continuous endomorphism with open image.
\end{itemize}
\end{prop}

\begin{center}{\bf Focal hyperbolic groups of mixed type}\end{center}\vspace{-0.1cm}

Finally, we say that a focal hyperbolic group is of {\em mixed type} if it is of neither connected nor totally disconnected type. For instance, if $p$ is prime and $\lambda\in\mathopen]0,1\mathclose[$, then the semidirect product $(\R\times\mathbf{Q}_p)\rtimes\Z$, where the positive generator of $\Z$ acts by multiplication by $(\lambda,p)$ on the ring $\R\times\mathbf{Q}_p$. 

Specifying once again \cite[Theorem 7.3]{CCMT}, we obtain

\begin{prop}\label{fomix}Let $G$ be a focal hyperbolic LC-group. Equivalences:
\begin{itemize}
\item $G$ is of mixed type;
\item the kernel of its modular function is neither connected-by-compact, nor compact-by-(totally disconnected);
\item it admits a continuous, proper cocompact isometric action on a {\em pure millefeuille space} (see \S\ref{sec:mi}).
\item it has a maximal compact normal subgroup $W$ such that $G/W$ has an open subgroup of finite index isomorphic to a semidirect product $(N_1\times N_2)\rtimes\Z$, where $\Z$ acts on $N_1\times N_2$ by compaction preserving the decomposition, $N_1$ is a connected Lie group, $N_2$ is totally disconnected, and $N_1$ and $N_2$ are both non-compact.
\end{itemize}
\end{prop}

As far as I know, it seems that focal groups of mixed type (including examples) were not considered 
before \cite{CoTe}.

\subsection{Millefeuille spaces and amenable hyperbolic groups}\label{sec:mi}
We here recall the definition of millefeuille spaces.


Given a metric space $X$, define a Busemann function $X\to\R$ as a limit, uniform on bounded subsets of $X$, of functions of the form $x\mapsto d(x,x_0)+c_0$ for $x_0\in X$ and $c_0\in\R$. By Busemann metric space, we mean a metric space $(X,b)$ endowed with a Busemann function; a shift-isometry of $(X,b)$ is an isometry $f$ of $X$ preserving $b$ up to adding constants, i.e.\ such that $x\mapsto b(f(x))-b(x)$ is constant. A homogeneous Busemann metric space means a Busemann metric space with a transitive group of shift-isometries.

Let $(X,b)$ be a complete CAT($\kappa$) Busemann space ($-\infty\le\kappa\le 0$). For $k$ a non-negative integer, let $T_k$ be a $(k+1)$-regular tree (identified with its 1-skeleton), endowed with a Busemann function denoted by $b'$ (taking integer values on vertices). Note that the Busemann space $(T_k,b')$ is uniquely determined by $k$ up to combinatorial shift-isometry. The millefeuille space $X[b,k]$, introduced in \cite[\S 7]{CCMT}, is by definition the topological space
\[\{(x,y)\in X\times T_k\mid b(x)=b'(y)\}.\]
Call {\em vertical geodesic} in $T_k$, a geodesic in restriction to which $b'$ is 
an isometry. Call {\em vertical leaf} in $X[b,k]$ a (closed) subset of the form $X[b,k]\cap (X\times V)$ where $V$ is a vertical geodesic. 
In \cite[\S 7]{CCMT}, it is observed that there is a canonical geodesic distance, defining the topology, and such that in restriction to any vertical leaf, the canonical projection to $X$ is an isometry. Moreover $X[b,k]$ is CAT($\kappa$), and is naturally a Busemann space, the Busemann function mapping $(x,y)$ to $b(x)=b'(y)$. Note that $X[b,0]=X$.

We have the following elementary well-known lemma
\begin{lem}\label{chca}
Let $X$ be a homogeneous connected negatively curved Riemannian manifold of dimension $\ge 2$. Then exactly one of the following holds:
\begin{enumerate}[(a)]
\item $\Isom(X)$ fixes a unique point in $\partial X$.
\item $X$ is a rank~1 symmetric space of noncompact type; in particular $\Isom(X)$ is transitive on $\partial X$. 
\end{enumerate}
In particular, $\Isom(X)$ has a unique closed orbit on $\partial X$, which is either a singleton or the whole $\partial X$.
\end{lem}
\begin{proof}
We use that $G=\Isom(X)$ is hyperbolic and the action of $G$ on $X$ is quasi-isometrically conjugate to the left action of $G$ on itself. If $G$ is amenable, it is focal hyperbolic and thus fixes a unique point $\omega_X$ on the boundary and $G$ is transitive on $X\smallsetminus \{\omega_X\}$ (see for instance \cite[Proposition 5.5(c)]{CCMT}), thus $\{\omega_X\}$ is the unique closed $G$-orbit. 

Otherwise, it is hyperbolic of general type and virtually connected, 
and thus, by a simple argument (see \cite[Proposition 5.10]{CCMT}), is isomorphic to an open subgroup in $\Isom(Y)$ for some rank 1 symmetric space $Y$ of noncompact type. Let $K\subset G$ be the stabilizer in $G$ of one point $x_0\in X$, by transitivity we have $X\simeq G/K$ as $G$-spaces. Since $X$ is CAT(-1) (up to homothety) and complete, $K$ is a maximal compact subgroup of $G$. Thus $K$ is also the stabilizer of one point in $Y$, and the identifications $X\simeq G/K\simeq Y$ then exchange $G$-invariant Riemannian metrics on $X$ and those on $Y$. Since on $Y$, $G$-invariant Riemannian metrics are unique up to scalar multiplication and are symmetric, this is also true on $X$. So $X$ is symmetric and in particular $\Isom(X)$ is transitive on $\partial X$.
\end{proof}

\begin{defn}
Let $X$ be a homogeneous simply connected negatively curved Riemannian manifold. A {\em distinguished boundary point} is a point in the closed $\Isom(X)$-orbit in $\partial X$ (see Lemma \ref{chca}).
A {\em distinguished Busemann function} is a Busemann function attached to a distinguished boundary point.
\end{defn}

\begin{lem}\label{disbu}
Let $X$ be a homogeneous simply connected negatively curved Riemannian manifold. For any two distinguished Busemann functions $b,b'$ on $X$, there exists an isometry from $(X,b)$ to $(X,b')$, i.e., there exists an isometry $f:X\to X'$ such that $b=b'\circ f$.
\end{lem}
\begin{proof}
Let $\omega$ and $\omega'$ be the distinguished points associated to $b$ and $b'$. Since all distinguished points are in the same $\Isom(X)$-orbit, we can push $b'$ forward by a suitable isometry and assume $\omega'=\omega$. Since the stabilizer in $\Isom(X)$ of $\omega$ is transitive on $X$, it is transitive on the set of Busemann functions attached to $\omega$. Thus there exists $f$ as required.
\end{proof}

Lemma \ref{disbu} allows to rather write $X[k]$ with no reference to any Busemann function, whenever $X$ is a homogeneous simply connected negatively curved Riemannian manifold.

The relevance of these spaces is due to the following theorem from \cite[\S 7]{CCMT}
\begin{thm}Let $G$ be a noncompact LC-group. Equivalences:
\begin{itemize}
\item $G$ is amenable hyperbolic;
\item for some integer $k\ge 1$ and some homogeneous simply connected negatively curved Riemannian manifold of positive dimension $d$, the group $G$ admits a continuous, proper and cocompact action by isometries on the millefeuille space $X[k]$, fixing a point (or maybe a pair of points if $(k,d)=(1,1)$) on the boundary.
\end{itemize}\label{ccmt_mil}
\end{thm}

Let us now describe the topology of the boundary of $\partial X[k]$.

\begin{prop}\label{bdmi}
Let $X$ be homogeneous negatively curved $d$-dimensional Riemannian manifold ($d\ge 1$) endowed with a surjective Busemann function and $k\ge 1$. Then $\partial X[k]$ is homeomorphic to the one-point compactification of
\begin{itemize}
\item $\R^{d-1}$ if $k=1$ and $d\ge 2$;
\item $\R^{d-1}\times\Z\times C$ if $k\ge 2$, where $C$ is a Cantor space.
\end{itemize}
In particular, its topological dimension is $d-1$.
\end{prop}
Note that when $d=1$ and $k\ge 2$, then $X[k]$ is a $(k+1)$-regular tree and  $\partial X[k]$ is homeomorphic to a Cantor space.

\begin{proof}[Proof of Proposition \ref{bdmi}]
We can pick a group $G_1=N_1\rtimes\Z$ acting properly cocompactly on $X$, shifting the Busemann function by integer values, such that the action of $\Z$ contracts the simply connected nilpotent $(d-1)$-dimensional Lie group $N_1$, and $G_2=N_2\rtimes\Z$ acting continuously transitively on the $(k+1)$-regular tree, where $\Z$ contracts $N_2$. 
Then $G=(N_1\times N_2)\rtimes\Z$ acts properly cocompactly on $X[k]$, and the action of $N_1\times N_2$ on $\partial X[k]\smallsetminus\{\omega\}$ is simply transitive (see the proof of \cite[Proposition 5.5]{CCMT}). So $X[k]\smallsetminus\{\omega\}$ is homeomorphic to $N_1\times N_2$ and thus $X[k]$ is its one-point compactification. (This proof makes use of homogeneity for the sake of briefness, but the reader can find a more geometric proof in a more general context.)
\end{proof}

\begin{cor}\label{corxkt}
The classification of the spaces $\partial X[k]$ up to homeomorphy is given by the following classes
\begin{itemize}
\item $k=1$, $d\ge 1$ is fixed: $\partial X[k]$ is a $(d-1)$-sphere;
\item $k\ge 2$ is not fixed, $d\ge 1$ is fixed: $\partial X[k]$ is a ``Cantor bunch" of $d$-spheres.
\end{itemize}
\end{cor}
\begin{proof}
The homeomorphy type of the space $\partial X[k]$ detects $d$, since $d-1$ is its topological dimension. Moreover, it detects whether $k=1$ or $k\ge 2$, because $\partial X[k]$ is locally connected in the first case and not in the second.
\end{proof}

\begin{cor}\label{fixmi}
Let $\omega$ be the distinguished point in $\partial X[k]$ (the origin of its distinguished Busemann function). Consider the action of $\Homeo(\partial X[k])$ on $\partial X[k]$. Then
\begin{itemize}
\item if $\min(k,d)=1$ then this action is transitive;
\item if $\min(k,d)\ge 2$ then this action has 2 orbits, namely the singleton $\{\omega\}$ and its complement $\partial X[k]\smallsetminus\{\omega\}$.
\end{itemize}
\end{cor}
\begin{proof}
The space $\partial X[k]\smallsetminus\{\omega\}$ is homogeneous under its self-homeomorphisms, and thus any transitive group of self-homeomorphisms extends to the one-point compactification. It remains to discuss whether $\omega$ belongs to the same orbit.
\begin{itemize}
\item If $\min(k,d)=1$, then $\partial X[k]$ is a Cantor space or a sphere and thus is homogeneous under its group of self-homeomorphisms.
\item If $\min(k,d)\ge 2$, and $x\in\partial X[k]$, then $X[k]\smallsetminus\{x\}$ is connected if and only if $x\neq\omega$. In other words, $\omega$ is the only cut-point in $\partial X[k]$ and thus is fixed by all self-homeomorphisms.\qedhere
\end{itemize}
\end{proof}

\begin{cor}\label{foc3}Among focal hyperbolic LC-groups, the three classes of focal groups of connected, totally disconnected and mixed type are closed under quasi-isometry.\end{cor}
\begin{proof}If $G$ is a focal hyperbolic LC-group, it is of connected type if and only if its boundary is locally connected, and of totally disconnected type if and only if its boundary is totally disconnected.
\end{proof}


\begin{remark}Another natural topological description of the visual boundary $\bd X[k]$ is that it is homeomorphic to the smash product $\partial X\wedge\partial T$ of the boundary of $X$ and the boundary of $T$, and is thus homeomorphic to the $(d-1)$-fold reduced suspension of $\partial T$, where $d=\dim(X)$. Actually, this allows to describe the whole compactification $\overline{X[k]}=X[k]\cup\partial X[k]$ as the smash product of $\overline{T}$ and a $(d-1)$-sphere.
\end{remark}

\begin{remark}
The dimension $\dim(X)$ can also be characterized as the asymptotic dimension of $X[k]$, giving another proof that it is a QI-invariant of $X[k]$.
\end{remark}

\begin{figure}[!t]
\includegraphics[scale=0.4]{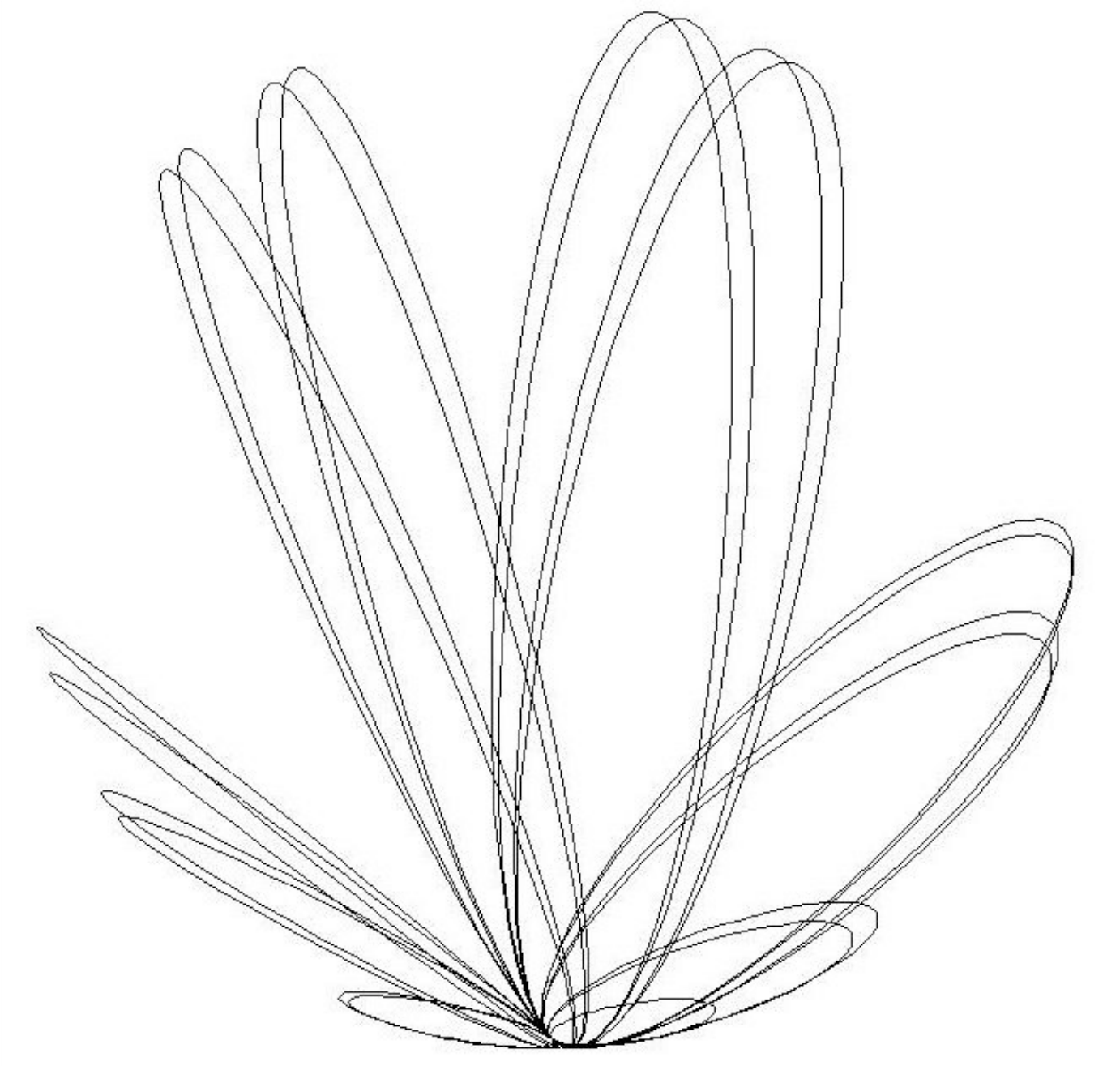}
\caption{The boundary of $\mathbf{H}^2_\R[2]$. Except at the singular point, it is locally modeled on the product of a dyadic Cantor space and a line.}
\label{figg}
\end{figure}

Let us call {\em pure millefeuille space} a millefeuille space $X[k]$ such that $\min(\dim X,k)\ge 2$, i.e.\ that is neither a manifold nor a tree; equivalently, its isometry group is focal of mixed type. 


\begin{prop}\label{mifi}
A hyperbolic LC-group $G$ is focal of mixed type if and only if $\partial G$ contains a point fixed by every self-homeomorphism of $\partial G$.
\end{prop}
\begin{proof} 
The ``only if" part follows from Corollary \ref{fixmi}

Conversely, assume that the condition is satisfied. Then $G$ cannot be of general type. The condition also implies that the boundary is nonempty and hence rules out compact groups. Otherwise $G$ is 2-ended or focal. But if $G$ is 2-ended, or focal of connected or totally disconnected type, its boundary is a nonempty sphere or a Cantor space, which have at least 2 elements and transitive homeomorphism groups.
\end{proof}


\begin{cor}\label{milnqi}
A focal hyperbolic LC-group of mixed type (or equivalently a pure millefeuille space) is not quasi-isometric to any hyperbolic LC-group of general type. In particular, it is not quasi-isometric to any vertex-transitive connected graph of finite valency.
\end{cor}
\begin{proof}
The first statement is a particular case of Proposition \ref{mifi}. For the second statement, observe that the isometry group of a such a vertex-transitive graph would be focal of totally disconnected type or of general type (by Lemma \ref{amqi}), but we have just excluded the general type case and the focal case is excluded by Corollary \ref{foc3}.
\end{proof}

\subsection{Carnot groups}\label{scarnot}

Recall that a {\em Heintze group} is a Lie group of the form $N\rtimes\R$, where $N$ is a simply connected nilpotent Lie group, and the action of positive reals contracts $N$. It is {\em purely real} if it only has real eigenvalues in the adjoint representation, or equivalently in the action of $\R$ on the Lie algebra of $N$.

\begin{defn}Let us say that a purely real Heintze group is of {\em Carnot type} if it admits a semidirect decomposition $N\rtimes\R$ such that, denoting by $(N^i)_{i\ge 1}$ the descending central series, it satisfies one of the three equivalent conditions
\begin{enumerate}
\item\label{ca1} the action of $\R$ on $N/[N,N]$ is scalar;
\item\label{ca2} the action of $\R$ on $N^i/N^{i+1}$ is scalar for all $i$;
\item\label{ca3} there is a linear decomposition of the Lie algebra $\mathfrak{n}=\bigoplus_{j=1}^\infty\mathfrak{v}_j$ such that, for some $\lambda\in\R\smallsetminus\{0\}$, the action of every $t\in\R$ on $\mathfrak{v}_j$ is given by multiplication by $\exp(j\lambda t)$ for all $j\ge 1$ and such that $\bigoplus_{j\ge i}\mathfrak{v}_j=\mathfrak{n}^i$.
\end{enumerate}
\label{d_carnot}\end{defn}

We need to justify the equivalence between the definitions. The trivial implications are (\ref{ca3})$\Rightarrow$(\ref{ca2})$\Rightarrow$(\ref{ca1}). To get (\ref{ca2}) from (\ref{ca1}), observe that as a module (for the action of the one-parameter subgroup), $\mathfrak{n}^i/\mathfrak{n}^{i+1}$ is a quotient of $(\mathfrak{n}/[\mathfrak{n},\mathfrak{n}])^{\otimes i}$, so if the action on $\mathfrak{n}/[\mathfrak{n},\mathfrak{n}]$ is scalar then so is the action on $\mathfrak{n}^i/\mathfrak{n}^{i+1}$. To get (\ref{ca3}) from (\ref{ca2}), use a characteristic decomposition of the action of the one-parameter subgroup.

\begin{remark}\label{grada}
If a purely real Heintze group $G=N\rtimes\R$ is of Carnot type then $N$ is a Carnot gradable\footnote{In the literature, ``Carnot gradable" is usually referred to as ``graded" but this terminology is in practice a source of confusion, inasmuch as nilpotent Lie algebras may admit other relevant gradings.} nilpotent group, in the sense that its Lie algebra admits a Carnot grading, i.e.\ a Lie algebra grading $\mathfrak{n}=\bigoplus_{i\ge 1}\mathfrak{v}_i$ such that $\mathfrak{n}^j=\bigoplus_{i\ge j}\mathfrak{v}_i$ for all $j$. The converse is not true: for instance most purely real Heintze groups of the form $\R^2\rtimes\R$ are not of Carnot type. Nevertheless, if $N$ is Carnot gradable, then up to isomorphism it defines a unique purely real Heintze group of Carnot type $\Carn(N)$ of the form $N\rtimes\R$. This is because a simply connected nilpotent Lie group $N$ is gradable if and only if its Lie algebra $\mathfrak{n}$ is isomorphic to the graded Lie algebra $\textnormal{grad}(\mathfrak{n})$ defined as $\bigoplus_{i\ge 1}\mathfrak{n}^i/\mathfrak{n}^{i+1}$ (where the bracket is uniquely defined by factoring the usual bracket $\mathfrak{n}^i\times\mathfrak{n^j}\to\mathfrak{n}^{i+j}$), so that if $\mathfrak{n}$ is Carnot gradable then any two Carnot gradings define the same graded Lie algebra up to graded isomorphism. 
\end{remark}


\section{Quasi-isometric rigidity of symmetric spaces}\label{s_ss}

\subsection{The QI-rigidity statement}

The general QI-rigidity statement for symmetric spaces of noncompact type is the following. Recall that a Riemannian symmetric space of noncompact type has a canonical decomposition as a product of irreducible factors; we call the metric {\em well-normalized} if all homothetic irreducible factors are isometric; this can always be ensured by a suitable factor-wise rescaling (for instance, requiring that the infimum of the sectional curvature on each factor is $-1$). 

\begin{thm}\label{cqi}
Let $X$ be a symmetric space of noncompact type with a well-normalized metric. Let $G$ be a compactly generated locally compact group, quasi-isometric to $X$. Then $G$ has a continuous, proper cocompact action by isometries on $X$.
\end{thm}

This statement is mostly known in a weaker form, where $G$ is assumed to be discrete and sometimes in an even much weaker form, where one allows to pass to a finite index subgroup. Still, it is part of a stronger result concerning arbitrary cocompact large-scale quasi-actions on $X$ of arbitrary groups (with no properness assumption) which, up to a minor continuity issue, is due to Kleiner and Leeb; see \S\ref{qapt}.

Before going into the proof, let us indicate some corollaries of Theorem \ref{cqi}. They illustrate the interest of having a statement for CGLC-groups instead of only finitely generated groups, even when studying discrete objects such as vertex-transitive graphs.

\begin{cor}\label{corgr}
Let $X$ be a symmetric space of noncompact type with a well-normalized metric. Let $Y$ be a vertex-transitive, finite valency connected graph. Assume that $Y$ is quasi-isometric to $X$. Then there exists a cocompact lattice $\Gamma\subset\Isom(X)$ and an extension $W\mono\tilde{\Gamma}\epi\Gamma$ with $W$ compact, such that $\tilde{\Gamma}$ admits a proper vertex-transitive action on $Y$; moreover there exists a connected graph $Z$ on which $\Gamma$ acts properly and transitively, with a surjective finite-to-one 1-Lipschitz equivariant quasi-isometry $Y\to Z$.
\end{cor}
\begin{proof}
Let $\tilde{\Gamma}$ be the automorphism group of $Y$; this is a totally disconnected LC-group; since $Y$ is vertex-transitive, it is compactly generated and quasi-isometric to $Y$ and  hence to $X$. By Theorem \ref{cqi}, there is a proper continuous homomorphism with cocompact image $\tilde{\Gamma}\to\Isom(X)$. Since $\Isom(X)$ is a Lie group, this homomorphism has an open kernel $K$, compact by properness. So the image is a cocompact lattice $\Gamma$.

To obtain $Z$, observe that the $K$-orbits in $Y$ have uniformly bounded diameter, so $Z$ is just obtained as the quotient of $Y$ by the $K$-action.
\end{proof}

\begin{cor}\label{corrm}
Let $X$ be a symmetric space of noncompact type with a well-normalized metric. Let $M$ be a proper, homogeneous geodesic metric space quasi-isometric to $X$. Then there exists a closed, connected cocompact subgroup $H$ of $\Isom(X)$ and an extension $W\mono \tilde{H}\epi H$ with $W$ compact, and a faithful proper isometric transitive action of $\tilde{H}$ on $M$. If moreover $M$ is contractible then $W=1$ and there is a diffeomorphism $M\to X$ intertwining the $H$-action on $M$ with the original $H$-action on $X$.
\end{cor}
\begin{proof}
The argument is similar to that of Corollary \ref{corgr}; here $\tilde{H}$ will be the identity component of the isometry group of $M$. This proves the first statement.

Fix a point $m_0\in M$ and let $K$ be its stabilizer in $\tilde{H}$. Then $K$ is a compact subgroup. The image of $K$ in $H$ fixes a point $x_0$ in $X$. So the homomorphism $\tilde{H}\to\Isom(X)$ induces a continuous map $j:M=\tilde{H}/K\to X$ mapping $m_0$ to $x_0$. Since $H$ is connected and closed cocompact, its action on $X$ is transitive; it follows that $j$ is surjective.

Under the additional assumption, let us check that $j$ is injective; by homogeneity, we have to check that $j^{-1}(\{x_0\})=\{m_0\}$. This amounts to check that if $g\in\tilde{H}$ fixes $x_0$ then it fixes $m_0$. Indeed, the stabilizer of $x_0$ in $\tilde{H}$ is a compact subgroup containing $K$; since $M$ is contractible, $K$ is maximal compact (see \cite{Anto}) and this implies the result. Necessarily $W=1$, because it fixes a point on $X$ and being normal it fixes all points, so also acts trivially on $M$; since $\tilde{H}$ acts by definition faithfully on $M$ it follows that $W=1$.
\end{proof}

\subsection{Quasi-actions and proof of Theorem \ref{cqi}}\label{qapt}
We use the large-scale language introduced in \S\ref{s_qii}.

The following powerful theorem is essentially due to Kleiner and Leeb \cite{kl09}, modulo the formulation, a continuity issue, and the fact we have equivariance instead of quasi-equivariance. It is stronger than Theorem \ref{cqi} because no properness assumption is required.

If $G$ is a locally compact group and acts, not necessarily continuously, by isometries on a metric space $X$, we say that the action is {\em locally bounded} if for some/every $x\in X$ and every compact subset $K\subset G$, the subset $Kx\subset X$ is bounded; in particular if the action is continuous, then it is locally bounded. We say that the action is {\em proper} if for every compact subset $B$ of $X$, the set $\{g\in G:\;gB\cap B\neq\emptyset\}$ has a compact closure in $G$. The action is {\em cobounded} if there exists $x\in X$ such that $\sup_{y\in X}d(y,Gx)<\infty$.

\begin{thm}\label{kl9}
Let $X$ be a symmetric space of noncompact type with a well-normalized metric (as defined before Theorem \ref{cqi}). Let $G$ be a locally compact group with a locally bounded and cobounded isometric action $\rho$ on a metric space $Y$ quasi-isometric to $X$. Then there exists an isometric action of $G$ on $X$ and a $G$-equivariant quasi-isometry $Y\to X$. Moreover, the latter action is necessarily continuous, cocompact, and if $\rho$ is proper then the action is proper.
\end{thm}

For the continuity issue, we need the following lemma.

\begin{lem}\label{bbcon}
Let $H$ be a locally compact group with a continuous isometric action on a metric space $X$. Assume that the only element of $H$ acting on $X$ as an isometry with bounded displacement is the identity. (For instance, $X$ is a minimal proper CAT(0)-space with no Euclidean factor and $H$ is the full isometry group.) Let $G$ be a locally compact group with an action of $G$ on $X$ given by a locally bounded homomorphism $\varphi \colon G\to H$.
Assume that the action of $G$ on $X$ is cobounded. 
Then $\varphi$ (and hence the action of $G$ on $X$) is continuous.
\end{lem}
\begin{proof}
Let $w\in H$ be an accumulation point of $\varphi(\gamma)$ when $\gamma\to 1$.
Fix compact neighborhoods $\Omega_G,\Omega_H$ of 1 in $G$ and $H$. Fix $g\in G$. There exists a neighborhood $V_g$ of $1$ in $G$ such that 
$\gamma\in V_g$ implies $g^{-1}\gamma g\in\Omega_G$. 
Moreover, if we define
\[C_g=\{\gamma\in G:\varphi(\gamma)^{-1}w\in\varphi(g)\Omega_H\varphi(g)^{-1}\},\]
then 1 belongs to the closure of $C_g$.
  Hence $V_g\cap C_g\neq\emptyset$, so if we pick $\gamma$ in the intersection we get
\[\varphi(g)^{-1}\varphi(\gamma)\varphi(g)\in\varphi(\Omega_G);\quad \varphi(g)^{-1}\varphi(\gamma)^{-1}w\varphi(g)\in\Omega_H,\]
so we deduce, for all $g\in G$
\[\varphi(g)^{-1}w\varphi(g)\in\varphi(\Omega_G)\Omega_H,\]
Since the action is locally bounded, $\varphi(\Omega_G)$ is bounded. This shows that for a given $x\in X$, the distance $d(w\varphi(g)x,\varphi(g)x)$ is bounded independently of $g$. Since $G$ acts coboundedly, this shows that $w$ has bounded displacement, so $w=1$. Since $H$ is locally compact, this implies that $\varphi$ is continuous at 1 and hence is continuous.
\end{proof}

We need to deduce Theorem \ref{cqi} from Theorem \cite[Theorem 1.5]{kl09}. For this we need to introduce all the terminology of quasi-actions. We define a {\em quasi-action} of a group $G$ on a set $X$ as an arbitrary map $G\to X^X$, written as $g\mapsto (x\mapsto gx)$. If $x\in X$, let $i_x:G\to X$ be the orbital map $g\mapsto gx$.

If $G$ is a group and $X$ a metric space, a {\em uniformly large-scale Lipschitz} (ULSL) quasi-action of $G$ on $X$ is a map $\rho:G\to X^X$, such that for some $(\mu,\alpha)\in\R_{>0}\times\R$, every map $\rho(g)$, for $g\in G$ is a $(\mu,\alpha)$-Lipschitz map (in the sense defined in \S\ref{s_qii}), and satisfying $\rho(1)\stackrel{\alpha}\sim\textnormal{id}$ and $\rho(gh)\stackrel{\alpha}\sim\rho(g)\rho(h)$ for all $g,h\in G$. Note that for a ULSL quasi-action, all $i_x$ are pairwise $\sim$-equivalent.

The quasi-action is {\em cobounded} if the map $i_x$ has a cobounded image for every $x\in X$. If $G$ is a locally compact group, a quasi-action is {\em locally bounded} if $i_x$ maps compact subsets to bounded subsets for all $x$ (if $G$ is discrete this is an empty condition). The quasi-action is {\em coarsely proper} if for all $x\in X$, inverse images of bounded subsets by $i_x$ have compact closure. 

Given quasi-actions $\rho$ and $\rho'$ of $G$ on $X$ and $X'$, a map $q:X\to X'$ is {\em quasi-equivariant} if it satisfies
\[\sup_{g\in G,x\in X}d\big(\,q(\rho(g)x),\rho(g)q(x)\,\big)<\infty.\]
The ULSL quasi-actions are quasi-isometrically quasi-conjugate if there is a quasi-equivariant quasi-isometry $X\to X'$; this is an equivalence relation.

\begin{lem}\label{qadico}
Let $(X,d)$ be a metric space and $G$ a group. Then
\begin{enumerate}[(a)]
\item\label{rh2} If $Y$ is a metric space with an isometric action of $G$ and with a quasi-isometry $q:Y\to X$, then there is a ULSL quasi-action of $G$ on $X$ so that $q$ is a quasi-isometric quasi-conjugacy.
\item\label{rh} Conversely, if $\rho$ is a ULSL quasi-action of $G$ on $X$, then there exists a metric space $Y$ with an isometric $G$-action and a quasi-equivariant quasi-isometry $Y\to X$.
\end{enumerate}
\end{lem}
\begin{proof}
We begin by the easier (\ref{rh2}). Let $s:X\to Y$ be a quasi-isometry inverse to $q$ and define a quasi-action of $G$ on $X$ by $\rho(g)x=q(g\cdot s(x))$; it is straightforward that this is a ULSL quasi-action and that $q$ is a quasi-isometric quasi-conjugacy.

For (\ref{rh}), first, let $Y\subset X$ be a maximal subset for the property that the $\rho(G)y$, for $y\in Y$, are pairwise disjoint. Let us show that $\rho(G)Y$ is cobounded in $X$. Indeed, denoting, by abuse of notation $a \approx b$ if $|a-b|$ is bounded by a constant depending only on the ULSL constants of $\rho$, if $x\in X$ by maximality there exists $y\in Y_0$ and $g\in G$ such that $\rho(g)x=\rho(h)y$. Then, denoting by $d_X$ the distance in $X$
\[d_X(x,\rho(g^{-1}h)y)\approx d_X(x,\rho(g^{-1})\rho(h)y)\approx d_X(\rho(g)x,\rho(h)y)=0, \]
so $\rho(G)Y$ is cobounded.

Consider the cartesian product $G\times Y$, whose elements we denote by $g\diamond y$ rather than $(g,y)$ for the sake of readability. Define a pseudometric on $G\times Y$ by 
\[d(g\diamond y,h\diamond z)=\sup_{k\in G}d_X(\rho(kg)y,\rho(kh)z)\]
If $G\times Y$ is endowed with the $G$-action $k\cdot(g\diamond y)=(kg\diamond y)$, then this pseudo-metric is obviously $G$-invariant. Consider the map $j:(g\diamond y)\mapsto \rho(g)y$. Let us check that $j$ is a quasi-equivariant quasi-isometry $(G\times Y,d)\to X$. We already know it has a cobounded image.
We obviously have
\begin{equation}\label{d0}d(g\diamond y,h\diamond z)\ge d_X(\rho(g)y,\rho (h)z)=d_X(j(g\diamond y),j(h\diamond z)),\end{equation}
and conversely, writing $a\preceq b$ if $a\le Cb+C$ where $C>0$ is a constant depending only of the ULSL constants of $\rho$
\begin{align*}d(g\diamond y,h\diamond z)= & \sup_{k\in G}d_X\big(\rho(kg)y,\rho(kh)z\big)\\\approx & \sup_{k\in G}d_X\big(\rho(k)\rho(g)y,\rho(k)\rho (h)z\big)\\
\preceq & d_X(\rho(g)y,\rho (h)z)=d_X(j(g\diamond y),j(h\diamond z)).
\end{align*}
This shows that $j$ is a large-scale bilipschitz embedding. Thus it is a quasi-isometry. It is also quasi-equivariant, as
\[d_X\big(j(h\cdot(g\diamond y)),\rho(h)j(g\diamond y)\big)=d_X\big(\rho(hg)y),\rho(h)\rho(g)y)\big)\approx 0.\]
So we obtain the conclusion, except that we have a pseudo-metric; if we identify points in $G\times Y$ at distance zero, the map $j$ factors through the quotient space, as follows from (\ref{d0}) and thus we are done.
\end{proof}

\begin{proof}[Proof of Theorem \ref{kl9}]
The statement of \cite[Theorem 1.5]{kl09} is the following: {\em Let $X$ be a symmetric space of noncompact type with a well-normalized metric. Let $G$ be a discrete group with a ULSL quasi-action $\rho$ on $X$. Then $\rho$ is QI quasi-conjugate to an isometric action.}

By Lemma \ref{qadico}, this can be translated into the following: {\em Let $X$ be as above and let $G$ be a discrete group with an isometric action $\rho$ on a metric space $Y$ quasi-isometric to $X$. Then there exists an isometric action of $G$ on $X$ and a $G$-quasi-equivariant quasi-isometry $q:Y\to X$.}

To get the theorem, we need a few improvements. First, we want the quasi-isometry to be equivariant (instead of quasi-equivariant). Starting from $q$ as above, we proceed as follows. Fix a bounded set $Y_0\subset Y$ containing one point in each $G$-orbit. For $y_0\in Y_0$, let $G_{y_0}$ be its stabilizer. Since $G_{y_0}\{y_0\}=\{y_0\}$ and $q$ is quasi-equivariant, we see that $G_{y_0}\{q(y_0)\}$ has its diameter bounded by a constant $C$ depending only on $q$ (and not on $y_0$). By the centre lemma, $G_{y_0}$ fixes a point at distance at most $C$ of $q(y_0)$, which we define as $q'(y_0)$. For $y\in Y$ arbitrary, we pick $g$ and $y_0\in Y_0$ ($y_0$ is uniquely determined) such that $gy_0=y$ and define $q'(y)=gq'(y_0)$. By the stabilizer hypothesis, this does not depend on the choice of $g$, and we see that $q'$ is at bounded distance from $q$. So $q'$ is a quasi-isometry as well, and by construction is $G$-equivariant.

Now assume that $G$ is locally compact. Applying the previous result to $G$ endowed with the discrete topology, we get all the non-topological conclusions. Now from the additional hypothesis that the action is locally bounded, we obtain that the action on $X$ is locally bounded. We then invoke Lemma \ref{bbcon}, using that $X$ has no non-trivial bounded displacement isometry, to obtain that the $G$-action on $X$ is continuous. It is clear that the action on $X$ is cobounded, and that if $\rho$ is metrically proper then the action on $X$ is proper.
\end{proof}

\begin{proof}[Proof of Theorem \ref{cqi}]
Fix a left-invariant word metric $d_0$ on $G$. Then $(G,d)$ is quasi-isometric to $X$ and the action of $G$ on $(G,d)$ is locally bounded. So by Theorem \ref{kl9}, there exists a continuous isometric proper cocompact action on $X$.
\end{proof}

Theorem \ref{cqi} was initially proved in the case of discrete groups, for symmetric spaces with no factor of rank one by Kleiner-Leeb \cite{kl} and in rank one using different arguments by:
\begin{itemize}
\item Pansu in the case of quaternionic and octonionic hyperbolic spaces;

\item R. Chow in the case of the complex hyperbolic spaces \cite{chow},
;
\item Tukia for real hyperbolic spaces of dimension at least three \cite{tukia86};
\item In the case of the real hyperbolic plane, Tukia \cite{tukia88}, completed by, independently Casson-Jungreis \cite{cj} and Gabai \cite{ga}.
\end{itemize}
The synthesis for arbitrary symmetric spaces is due to Kleiner-Leeb \cite{kl09}, where they insightfully consider actions of arbitrary groups, not considering any topology but with no properness assumption. This generality is essential because when considering a proper cocompact locally bounded action of a non-discrete LC-group, when viewing it as an action of the underlying discrete group, we lose the properness.

\section{Quasi-isometric rigidity of trees}\label{qitree}

\subsection{The QI-rigidity statement}\label{qista}

The next theorem is the locally compact version of the result that a finitely generated group is quasi-isometric to a tree if and only if it is virtually free, which is a combination of Bass-Serre theory \cite{Ser77}, Stallings' theorem and Dunwoody's result that finitely presented groups are accessible \cite{du85}, a notion we discuss below.

\begin{thm}\label{cqia}
Let $G$ be a compactly generated locally compact group. Equivalences:
\begin{enumerate}
\item\label{cqia1b} $G$ is quasi-isometric to a tree;
\item\label{cqia1} $G$ is quasi-isometric to a bounded valency tree;
\item\label{cqia2} $G$ admits a continuous proper cocompact isometric action on a locally finite tree $T$, or a continuous proper transitive isometric action on the real line;
\item\label{cqia3} $G$ is topologically isomorphic to the Bass-Serre fundamental group of a finite connected graph of groups, with compact vertex and edge stabilizers with open inclusions, or has a continuous proper transitive isometric action on the real line.
\end{enumerate}
\end{thm}

\begin{proof}[Proof (first part)]
Let us mention that the equivalence between (\ref{cqia2}) and (\ref{cqia3}) is just Bass-Serre theory \cite{Ser77}. Besides, the trivial implications are (\ref{cqia2})$\Rightarrow$(\ref{cqia1})$\Rightarrow$(\ref{cqia1b}). 

Let us show (\ref{cqia1b})$\Rightarrow$(\ref{cqia1}); let $G$ be compactly generated and let $f:T\to G$ be a quasi-isometry from a tree $T$ (viewed as its set of vertices). We can pick a metric lattice $J$ inside $G$ and assume that $f$ is valued in $J$; the metric space $J$ has finite balls of cardinality bounded by a constant depending only on their radius. We can modify $f$ to ensure that $f^{-1}(\{j\})$ is convex for every $j\in J$. Since $f$ is a quasi-isometry the convex subsets $f^{-1}(\{j\})$ have uniformly bounded radius. We define a tree $T''$ by collapsing each convex subset $f^{-1}(\{j\})$ (along with the edges joining them) to a point. The collapsing map $T'\to T''$ is a 1-Lipschitz quasi-isometry and thus $f$ factors through an injective quasi-isometry $T''\to G$. It follows that $T''$ has balls of cardinal bounded in terms of the radius; this implies in particular that $T''$ has finite (indeed bounded) valency. 

We postpone the proof of (\ref{cqia1})$\Rightarrow$(\ref{cqia3}), which is the deep part of the theorem.
\end{proof}

\begin{remark}
Another characterization in Theorem \ref{cqia} is the following: $G$ is of type $\textnormal{FP}_2$ and has asymptotic dimension $\le 1$. Here type $\textnormal{FP}_2$ means that the homology of the Cayley graph of $G$ with respect to some compact generating subset is generated by loops of bounded length. Indeed, Fujiwara and Whyte \cite{FW} proved that a geodesic metric space satisfying these conditions is quasi-isometric to a tree.
\end{remark}

By {\em essential tree} we mean a nonempty tree with no vertex of degree 1 and not reduced to a line. By {\em reduced} graph of groups we mean a connected graph of groups in which no vertex group is trivial, and such that for every oriented non-self edge, the target map is not surjective. We call a reduced graph {\em nondegenerate} if its Bass-Serre tree is not reduced to the empty set, a point or a line, or equivalently if it is not among the following exceptions:
\begin{itemize}
\item The empty graph;
\item A single vertex with no edge;
\item A single vertex and a single self-edge so that both target maps are isomorphisms;
\item 2 vertices joined by a single edge, so that both target maps have image of index 2.
\end{itemize}

\begin{cor}\label{cqiab}
Let $G$ be a compactly generated locally compact group. Equivalences:
\begin{itemize}
\item[(\ref{cqia1b})] $G$ is quasi-isometric to an unbounded tree not quasi-isometric to the real line;
\item[(\ref{cqia1})] $G$ is quasi-isometric to the 3-regular tree;
\item[(\ref{cqia2})] $G$ admits a continuous proper cocompact action on an essential locally finite nonempty tree $T$; 
\item[(\ref{cqia2}')] $G$ admits a continuous proper cocompact isometric action on nonempty tree $T$ of bounded valency, with only vertices of degree $\ge 3$; 
\item[(\ref{cqia3})] $G$ is topologically isomorphic to the Bass-Serre fundamental group of a finite nondegenerate reduced graph of groups, with compact vertex and edge stabilizers, and open inclusions.\qed
\end{itemize}
\end{cor}

\subsection{Reminder on the space of ends}
Recall that the set of ends of a geodesic metric space $X$ is the projective limit of $\pi_0(X-B)$, where $B$ ranges over bounded subsets of $X_1$ and $\pi_0$ denotes the set of connected components. Each $\pi_0(X-B)$ being endowed with the discrete topology, the set of ends is endowed with the projective limit topology and thus is called the space of ends $E(X)$.

If $X,Y$ are geodesic metric spaces, any large-scale Lipschitz, coarsely proper map $X\to Y$ canonically defines a continuous map $E(X)\to E(Y)$. Two maps at bounded distance induce the same map. This construction is functorial. In particular, it maps quasi-isometries to homeomorphisms.

If $G$ is a locally compact group generated by a compact subset $S$, the set of ends \cite{Sp,Ho} of $G$ is by definition the set of ends of the 1-skeleton of its Cayley graph with respect to $S$. It is compact. Since the identity map $(G,S_1)$ to $(G,S_2)$ is a quasi-isometry, the space of ends is canonically independent of the choice of $S$ and is functorial with respect to continuous proper group homomorphisms.

In the case when $G$ is hyperbolic, it is not hard to prove that the space of ends is canonically homeomorphic to the space $\pi_0(\bd G)$ of connected components of the visual boundary.

We note for reference the following elementary lemma, due to Houghton. It is a particular case of \cite[Th.~4.2 and 4.3]{Ho}.
\begin{lem}\label{hough}
Let $G$ be a CGLC-group with $G_0$ noncompact. Then $G$ is (1~or~2)-ended; if moreover $G/G_0$ is not compact then $G$ is 1-ended.
\end{lem}

Let us also mention, even if it will be not be used here, the following theorem of Abels \cite{Ab77}: if a compactly generated LC-group has at least 3 ends, the action of $G$ on $E(G)$ is minimal (i.e., all orbits are dense) unless $G$ is focal hyperbolic of totally disconnected type (in the sense of \S\ref{fhtd}).

\subsection{Metric accessibility}

\begin{defn}[Thomassen, Woess \cite{tw}]Let $X$ be a bounded valency connected graph. Say that $X$ is {\em accessible} if there exists $m$ such that for every two distinct ends of $X$, there exists an $m$-element subset of the 1-skeleton $X$ that separates the two ends, i.e.\ so that the two ends lie in distinct components of the complement. 
\end{defn}

Accessibility is a quasi-isometry invariant of connected graphs of bounded valency.
Although not needed in view of Theorem \ref{cqia}, we introduce the following definition, which is more metric in nature.

\begin{defn}Let $X$ be a geodesic metric space. Let us say that $X$ is {\em diameter-accessible} if there exists $m$ such that for every two distinct ends of $X$, there exists a subset of $X$ of diameter at most $m$ that separates the two ends.
\end{defn}

This is obviously a quasi-isometric invariant property among geodesic metric spaces. Obviously for a bounded valency connected graph, diameter-accessibility implies accessibility. The reader can construct, as an exercise, a connected planar graph of valency $\le 3$ that is accessible (with $m=2$) but not-diameter accessible.

\begin{thm}[Dunwoody]
Let $X$ be a connected, locally finite simplicial 2-complex with a cocompact isometry group. Assume that $H^1(X,\Z/2\Z)=0$ (e.g., $X$ is simply connected). Then $X$ is diameter-accessible.
\label{gdun}\end{thm}
\begin{proof}[On the proof]
This statement is not explicit, but is the contents of the proof of \cite{du85} (it is quoted in \cite[Theorem 15]{MSW} in a closer way). This proof consists in finding, denoting by $G=\Aut(X)$, an equivariant family $(C_i)_{i\in I}$, indexed by a discrete $G$-set $I$ with finitely many $G$-orbits, of pairwise disjoint compact subsets of $X$ homeomorphic to graphs and each separating $X$ (each $X\smallsetminus C_i$ is not connected), called tracks, so that each component of $X\smallsetminus \bigcup C_i$ has the property that its stabilizer in $G$ acts coboundedly on it, and is (at most 1)-ended. Thus any two ends of $X$ are separated by one of these components, which have uniformly bounded diameters.
\end{proof}

\subsection{The locally compact version of the splitting theorem}
The following theorem is the locally compact version of Stallings' theorem; it was proved by Abels (up to a minor improvement in the 2-ended case).

\begin{thm}[Stallings, Abels]\label{struktur}Let $G$ be a compactly generated, locally compact group. Then $G$ has at least two ends if and only if splits as a non-trivial HNN-extension or amalgam over a compact open subgroup, unless $G$ is 2-ended and is compact-by-$\R$ or compact-by-$\Isom(\R)$.\end{thm}

\begin{proof}[On the proof]
The main case is when $G$ is a closed cocompact isometry group of a vertex-transitive graph; this is \cite[Struktursatz 5.7]{Abe}. He deduces the theorem \cite[Korollar 5.8]{Abe} when $G$ has at least 3 ends in \cite[Struktursatz 5.7]{Abe}, the argument also covering the case when $G$ has a compact open subgroup (i.e.\ $G_0$ is compact) in the case of at least 2 ends.

Assume now that $G_0$ is noncompact. By Lemma \ref{hough}, if $G/G_0$ is noncompact then $G$ is 1-ended, so assume that $G/G_0$ is compact.
Then $G$ has a maximal compact subgroup $K$ so that $G/K$ admits a $G$-invariant structure of  Riemannian manifold diffeomorphic to a Euclidean space (\cite[Theorem 3.2]{Mos} and \cite[Theorem 4.6]{MZ}); the only case where this manifold has at least two ends is when it is one-dimensional, hence isometric to $\R$, whence the conclusion. 
\end{proof}

This yields the following corollary, which we state for future reference. The characterization (\ref{2e2}) of 2-ended groups is contained in Houghton \cite[Theorem~3.7]{Ho}; the stronger characterization (\ref{2e22}) is due to Abels \cite[Satz~B]{Abe}; the stronger versions (\ref{2e3}) or (\ref{2e4}) can be deduced directly without difficulty; they are written in \cite[Proposition~5.6]{CCMT}.

\begin{cor}\label{2ended}Let $G$ be a compactly generated, locally compact group. Equivalences:
\begin{enumerate}
\item\label{2e1} $G$ has exactly 2 ends;
\item\label{2bouts} $G$ is hyperbolic and $\#\partial G=2$;
\item\label{2e11} $G$ is quasi-isometric to $\Z$;
\item\label{2e2} $G$ has a discrete cocompact infinite cyclic subgroup;
\item\label{2e22} $G$ has an open subgroup of index $\le 2$ admitting a continuous homomorphism with compact kernel onto $\Z$ or $\R$;
\item\label{2e3} $G$ admits a continuous proper cocompact isometric action on the real line;
\item\label{2e4} $G$ has a (necessarily unique) compact open normal subgroup $W$ such that $G/W$ is isomorphic to one of the 4 following groups: $\Z$, $\Isom(\Z)$, $\R$, $\Isom(\R)$.
\end{enumerate}
\end{cor}
\begin{proof}
The implications (\ref{2e4})$\Rightarrow$(\ref{2e3})$\Rightarrow$(\ref{2e22})$\Rightarrow$(\ref{2e2})$\Rightarrow$(\ref{2e11})$\Rightarrow$(\ref{2bouts})$\Rightarrow$(\ref{2e1}) are clear. Assume (\ref{2e1}). Then, by Theorem \ref{struktur}, either (\ref{2e4}) holds explicitly or $G$ splits as a non-trivial HNN-extension or amalgam over a compact open subgroup; the only cases for which this does not yield more than 2 ends is the case of a degenerate HNN extension (given by an automorphism of the full vertex group) or an amalgam over a subgroup of index 2 in both factors. Thus $G$ maps with compact kernel onto $\Z$ or the infinite dihedral group $\Isom(\Z)$. 
\end{proof}

\subsection{Accessibility of locally compact groups}

It is natural to wonder whether the process of applying iteratively Theorem \ref{struktur} stops. This motivates the following definition.

\begin{defn}\label{gac}
A CGLC-group $G$ is {\em accessible} if it has a continuous proper transitive isometric action on the real line, or 
satisfies one of the following two equivalent conditions
\begin{enumerate}
\item either $G$ admits a continuous cocompact action on a locally finite tree $T$ with (at most 1)-ended vertex stabilizers and compact edge stabilizers;
\item $G$ is topologically isomorphic to the Bass-Serre fundamental group of a finite graph of groups, with compact edge stabilizers with open inclusions and (at most 1)-ended vertex groups.
\end{enumerate}
\end{defn}
The equivalence between the two definitions is Bass-Serre theory \cite{Ser77}. For either definition, trivial examples of accessible CGLC groups are (at most 1)-ended groups. Also, by Theorem \ref{struktur}, 2-ended CGLC groups are accessible, partly using the artifact of the definition. A non-accessible finitely generated group was constructed by Dunwoody in \cite{Du93}, disproving a long-standing conjecture of Wall.

\subsection{Metric vs group accessibility and proof of Theorem \ref{cqia}}
The metric notion of accessibility, which unlike in this paper was introduced after group accessibility, allowed Thomassen and Woess, using the Dicks-Dunwoody machinery \cite{DiDu}, to have a purely geometric characterization of group accessibility (as defined in Theorem \ref{cqia}). The following theorem is the natural extension of their method to the locally compact setting, due to Kr\"on and M\"oller \cite[Theorem~15]{KM}.

\begin{thm}\label{km}
Let $X$ be a bounded valency nonempty connected graph and $G$ a locally compact group acting continuously, properly cocompactly on $X$ by graph automorphisms. Then the following are equivalent
\begin{enumerate}
\item\label{ac1} $G$ is accessible (as defined in Definition \ref{gac});
\item\label{ac2} $X$ is accessible;
\item\label{ac3} $X$ is diameter-accessible.
\end{enumerate}
\end{thm}
\begin{proof}[On the proof]
The statement in \cite{KM} is the equivalence between (\ref{ac1}) and (\ref{ac2}) when $G$ is totally disconnected. But actually the easy implication (\ref{ac1})$\Rightarrow$(\ref{ac2}) yields (\ref{ac1})$\Rightarrow$(\ref{ac3}) without change in the proof, as they check that $X$ is quasi-isometric to a bounded valency graph with a cocompact action, with a $G$-equivariant family of ``cuts", which are finite sets, with finitely many $G$-orbits of cuts, so that any two distinct ends are separated by one cut. Finally  (\ref{ac3})$\Rightarrow$(\ref{ac2}) is trivial for an arbitrary bounded valency graph.
\end{proof}

\begin{proof}[End of the proof of Theorem \ref{cqia}]
It remains to show the main implication of Theorem \ref{cqia}, namely that (\ref{cqia1}) implies (\ref{cqia3}).  
Let $G$ be a CGLC-group quasi-isometric to a tree $T$. We begin by the claim that $G$ has no 1-ended closed subgroup. Indeed, if $H$ were such a group, then, being noncompact $H$, admits a bi-infinite geodesic, and using the quasi-isometry to $T$ we see that this geodesic has 2 distinct ends in $G$. Then these two ends are distinct in $H$, a contradiction.
Let us now prove (\ref{cqia3}). 
We first begin by three easy cases
\begin{itemize}
\item $G$ is compact. There is nothing to prove.
\item $G$ is 1-ended. This has just been ruled out.
\item $G$ is 2-ended. In this case Corollary \ref{2ended} gives the result.
\end{itemize}
So assume that $G$ has at least 3 ends. By Lemma \ref{hough}, $G_0$ is compact, so by Proposition \ref{tdg}, it admits a continuous proper cocompact isometric action on a connected finite valency graph $X$. Since $X$ is quasi-isometric to a tree, it is diameter-accessible. By Theorem \ref{km}, we deduce that $G$ is accessible. This means that it is isomorphic to the Bass-Serre fundamental group of a graph of groups with (at most 1)-ended vertex groups and compact open edge groups. As we have just shown that $G$ has no closed 1-ended subgroup, it follows that vertex groups are compact. So (\ref{cqia3}) holds.
\end{proof}

\subsection{Cobounded actions}

There is a statement implying Theorem \ref{cqia}, essentially due to Mosher, Sageev and Whyte, concerning cobounded actions with no properness assumption. In a tree, say that a vertex is an {\em essential branching vertex} if its complement in the 1-skeleton has at least 3 unbounded components. Say that a tree is {\em bushy} if the set of essential branching vertices is cobounded. We here identify any tree to its 1-skeleton.

\begin{thm}\label{mMSW}
Let $G$ be a locally compact group. Consider a locally bounded, isometric, cobounded action of $G$ on a metric space $Y$ quasi-isometric to a bushy tree $T$ of bounded valency. Then there exists a tree $T'$ of bounded valency with a continuous isometric action of $G$ and an equivariant quasi-isometry $T\to T'$. 
\end{thm}

\begin{remark}
It is not hard to check that a bounded valency non-bushy tree $T$ quasi-isometric to a metric space with a cobounded isometry group, is either bounded or contains a cobounded bi-infinite geodesic.
\end{remark}

\begin{remark}
Theorem \ref{mMSW} is not true when $T$ is a linear tree (the Cayley graph of $(\Z,\{\pm 1\})$). Indeed, taking $G$ to be $\R$ acting on $\R$ (which is quasi-isometric to $T$), there is no isometric cobounded action on any tree quasi-isometric to $\Z$. Indeed this action would preserve a unique axis, while there is no nontrivial homomorphism from $\R$ to $\Isom(\Z)$. In this case, we can repair the issue by allowing actions on $\R$-trees. But here is a second more dramatic counterexample: the universal covering $G=\widetilde{\SL}_2(\R)$ (endowed with either the discrete or Lie topology) admits a locally bounded isometric action on the Cayley graph of $(\R,[-1,1])$ (see \cite[Example 3.12]{CCMT}), but admits no isometric cobounded action on any $\R$-tree quasi-isometric to $\Z$, because by the same argument, this action would preserve an axis, while there is no nontrivial homomorphism of abstract groups $\widetilde{\SL}_2(\R)\to\Isom(\R)$.
\end{remark}

\begin{proof}[On the proof of Theorem \ref{mMSW}]
The statement of \cite[Theorem 1]{MSW} is in terms of quasi-actions, see the conventions in \S\ref{qapt}. It reads: {\em Let $G$ be a discrete group. Suppose that $G$ has a ULSL quasi-action on a bushy tree $T$ of bounded valency. Then there exists a tree $T'$ of bounded valency with a continuous isometric action of $G$ and a quasi-equivariant quasi-isometry $T\to T'$.}

By Lemma \ref{qadico}, it can be translated as: {\em (*) Let $G$ be a discrete group. Consider a locally bounded, isometric, cobounded action of $G$ on a metric space $Y$ quasi-isometric to a bushy tree $T$ of bounded valency. Then there exists a tree $T'$ of bounded valency with an isometric action of $G$ and a quasi-equivariant quasi-isometry $T\to T'$.}

To get the statement of Theorem \ref{mMSW}, apply (*) to the underlying discrete group; the action on $T'$ we obtain is then locally bounded. Since $T'$ has at least 3 ends, the only isometry with bounded displacement is the identity, so we deduce by Lemma \ref{bbcon} that the action on $T'$ is continuous.

Let us now sketch the proof of (*) (which is the discrete case of Theorem \ref{mMSW}), as the authors of \cite{MSW} did not seem to be aware of the Thomassen-Woess approach and repeat a large part of the argument.

Start from the hypotheses of (*). A trivial observation is that we can suppose $Y$ to be a connected graph. Namely, using that $T$ is geodesic, there exists $r$ such that if we endow $Y$ with a graph structure by joining points at distance $\le r$ by an edge, then the graph is connected and the graph metric and the original metric on $Y$ are quasi-isometric through the identity. The difficulty is that $Y$ need not be of finite valency.

The construction of an isometric action of $G$ on a connected {\em finite valency} graph $Z$ quasi-isometrically quasi-conjugate to the original quasi-action is done in \cite[\S 3.4]{MSW} (this is the part where it is used that the tree is bushy). It is given by a homomorphism $G\to\Isom(Z)$. By Theorem \ref{cqia}, $\Isom(Z)$ has a continuous proper cocompact action on a tree $T'$. Note that the actions of $\Isom(Z)$ on both $Z$ and $T'$ are quasi-isometrically conjugate to the left action of $\Isom(Z)$ on itself, so there exists a quasi-isometry $Z\to T'$ which is quasi-equivariant with respect to the $\Isom(Z)$-action, and therefore is quasi-equivariant with respect to the $G$-action. Finally, to get equivariance instead of quasi-equivariance, we use the same argument (based on the center lemma, which holds in $T'$) as in the proof of Theorem \ref{kl9}.
\end{proof}

\subsection{Accessibility of compactly presented groups}\label{acp}

Recall that a CGLC-group is {\em compactly presented} if has a presentation with a compact generating set and relators of bounded length. If it has a compact open subgroup, it is equivalent to require that $G$ has a continuous proper cocompact combinatorial action on a locally finite simply connected simplicial 2-complex (see Proposition \ref{cara}).

Let us also use the following weaker variant: $G$ is of type $\textnormal{FP}_2$ mod $2$ if 
$G$ is quotient of a compactly presented LC-group by a discrete normal subgroup $N$ such that $\Hom(N,\Z/2\Z)=0$. If $G$ has a compact open subgroup, it is equivalent (see Proposition \ref{cara2}) to require that $G$ has a continuous proper cocompact combinatorial action on a locally finite connected 2-complex $X$ such that $H^1(X,\Z/2\Z)=0$.

By Lemma \ref{hough}, if $G_0$ is not compact then $G$ is (at most two)-ended and thus is accessible by Corollary \ref{2ended}. Otherwise if $G_0$ is compact (i.e., $G$ has a compact open subgroup), the restatement of the definition shows that as a consequence of Theorems \ref{gdun} and \ref{km}, we have the following theorem, whose discrete version was the main theorem in \cite{du85}.

\begin{thm}
If $G$ is a compactly generated locally compact group of type $\textnormal{FP}_2$ mod 2 (e.g., $G$ is compactly presented) then $G$ is accessible. \qed
\end{thm}

\begin{cor}
Every hyperbolic locally compact group is accessible.\qed
\end{cor}

\begin{prop}\label{cara}
Let $G$ be a locally compact group with $G_0$ compact. Equivalences:
\begin{enumerate}
\item\label{cp} $G$ is compactly presented, that it, it has a presentation with a compact generating set and relators of bounded length;
\item\label{cp3} $G$ has a presentation with a compact generating set and relators of length $\le 3$;
\item\label{scc} $G$ has a continuous proper cobounded combinatorial action on a simply connected simplicial 2-complex $X$.
\item\label{scclf} $G$ has a continuous proper cocompact combinatorial action on a locally finite simply connected simplicial 2-complex $X$.
\end{enumerate}
\end{prop}
\begin{proof}
Trivially we have (\ref{scclf})$\Rightarrow$(\ref{scc}) and (\ref{cp3})$\Rightarrow$(\ref{cp}).

(\ref{scc})$\Rightarrow$(\ref{cp3}). Let $A$ be a connected bounded closed subset of $X$
such that the $G$-translates of $A$ cover $X$. Then if $S=\{g\in G\mid gA\cap A\neq\emptyset\}$, then $S$ is compact by properness, and by \cite[Lemma p.~26]{Kos}, $G$ is generated by $S$ and relators of length $\le 3$. So (\ref{cp3}) holds. 

(\ref{cp})$\Rightarrow$(\ref{scclf}). Let $S$ be a compact symmetric generating subset with relators of bounded length; enlarging $S$ if necessary, we can suppose that $S$ contains a compact open subgroup $K$ and that $S=KSK$. The Cayley-Abels graph (Proposition \ref{tdg}) is a graph with vertex set $G/K$ and an unoriented edge from $gK$ to $g'K$ if $g^{-1}g'\in S\smallsetminus K$. This is a $G$-invariant locally finite connected graph structure on $G/K$. Define a simplicial 2-complex structure by adding a triangle every time its 2-skeleton appears. 

Define a simplicial 2-complex structure on the Cayley graph of $G$ with respect to $S$ in the same way. Then the projection $G\to G/K$ has a unique extension between the 2-skeleta, affine in each simplex (possibly collapsing a simplex onto a simplex of smaller dimension). The Cayley 2-complex of $G$ is simply connected, because the relators of length $\le 3$ define the group.  

It remains to check that the simplicial 2-complex on $G/K$ is simply connected. Indeed, given a loop based at 1, we can homotope it to a combinatorial loop on the 1-skeleton. We can lift such a loop to a combinatorial path $(1=g_0,g_1,\dots,g_k)$ on $G$ and ending at the inverse image of 1, namely $K$. If $k=0$ there is nothing to do, otherwise observe that $g_{k-1}\in KS\subset S$. So we can replace $g_k$ by 1, thus the new lifted path is a loop. This loop has a combinatorial homotopy to the trivial loop, which can be pushed forward to a combinatorial homotopy of the loop on $G/K$. So the  simplicial 2-complex on $G/K$ is simply connected. Thus (\ref{scclf}) holds.
\end{proof}

\begin{lem}\label{caraf}
Let $G$ be a compactly presented locally compact group with $G_0$ compact. Then $G$ has a continuous proper cocompact combinatorial action on a locally finite simply connected simplicial 2-complex $X$ with no inversion (i.e.\ for each $g\in G$, the set of $G$-fixed points is a subcomplex) and with a main vertex orbit, in the sense that every $g\in G$ fixing a vertex, fixes a vertex in the main orbit, and if an element is the identity on the main orbit then it is the identity. Moreover, the main vertex orbit can be chosen as $G/K$ for any choice $K$ of compact open subgroup.
\end{lem}
\begin{proof}[Sketch of proof]
The proof is essentially a refinement of the proof of
 (\ref{cp})$\Rightarrow$(\ref{scclf}) of Proposition \ref{cara}, so we just indicate how it can be adapted to yield the lemma. We start with the same construction, but using oriented edges (without self-loops), and then we add, each time we have 3 oriented edges $(e_1,e_2,e_3)$ forming a triangle with 3 distinct vertices (with compatible orientations, i.e.\ the target of $e_1$ is the source of $e_2$, etc.), we add a triangle. Because of the double edges, this is not yet simply connected; so for any two adjacent vertices $x,y$ we glue two bigons indexed by $(x,y)$ and $(y,x)$ (note that the union of these 2 bigons is homeomorphic to a 2-sphere). The resulting complex is simply connected. Then the action of $G$ has the required property. The problem is that because of bigons, we do not have a simplicial complex. To solve this, we just add vertices at the middle of all edges and all bigons, split the bigons into 4 triangles by joining the center to all 4 vertices (thus pairs of opposite bigons now form the 2-skeleton of a octahedron), and split the original triangles into 4 triangles by joining the middle of the edges. We obtain a simplicial 2-complex; the cost is that we have several vertex orbits, but then $G/K$ is the ``main orbit" in the sense of the proposition.
\end{proof}

\begin{prop}\label{cara2} 
Let $G$ be a locally compact group with $G_0$ compact. Let $A$ be a discrete abelian group. Equivalences:
\begin{enumerate}
\item\label{cna} $G$ is isomorphic to a quotient $\tilde{G}/N$ with $\tilde{G}$ compactly presented, $N$ a discrete normal subgroup of $\tilde{G}$, and $\Hom(N,A)=0$;
\item\label{lfh} $G$ has a continuous proper cocompact combinatorial action on a locally finite simplicial 2-complex $X$ with $H^1(X,A)=0$.
\end{enumerate}
\end{prop}
\begin{proof}
Suppose (\ref{cna}) and let us prove (\ref{lfh}). We can suppose $G_0=\{1\}$.
Choose $K$ to be a compact open subgroup with $K\cap N=\{1\}$. Consider an action $\alpha$ of $\tilde{G}$ on a 2-complex $X$ as in Lemma \ref{caraf} and main vertex orbit $G/K$. Then the action of $N$ on $X$ is free: indeed, if an element of $N$ fixes a point, then it fixes a vertex in the main orbit and hence it has some conjugate in $N\cap K$, thus is trivial.


So $G=\tilde{G}/N$ acts continuously properly cocompactly on $N\backslash X$, which is a simplicial 2-complex as $N$ acts freely on $X$. Now we have, fixing an implicit base-point in $X$
\begin{equation}\label{h1}H^1(X,A)\simeq\Hom(\pi_1(N\backslash X),A)\end{equation}
Since $\pi_1(N\backslash X)\simeq N$, we deduce that $H^1(X,A)=0$.

Conversely suppose (\ref{lfh}), denote by $\alpha$ the action of $G$ on $X$ and $W$ its kernel. Fixing a base-vertex in $X$, let $\tilde{X}$ be the universal covering. Let $H$ be the group of automorphisms of $\tilde{X}$ that induce an element of $\alpha(G)$. It is a closed subgroup and is cocompact on $\tilde{X}$; its projection $\rho$ to $\alpha(G)$ is surjective. Let
\[\tilde{G}=\{(g,h)\in G\times H\mid \alpha(g)=\rho(h)\}\subset G\times H\] be the fibre product of $G$ and $H$ over $\alpha(G)$. Then it contains $W\times\{1\}$ as a compact normal subgroup, and the quotient is canonically isomorphic to $H$. The kernel of the projection $\tilde{G}\to G$ is equal to $\{1\}\times\Ker(\rho)$, which is discrete and consists of deck transformations of the covering $\tilde{X}\to X$ and is in particular isomorphic to $\pi_1(X)$. Again using (\ref{h1}), we obtain that $\Hom(N,A)=\{0\}$.
\end{proof}


\begin{remark}
In case $A$ is the trivial group, both (\ref{cna}) and (\ref{lfh}) of 
Proposition \ref{cara2} hold.
At the opposite, keeping in mind that any nontrivial abelian group $B$ satisfies $\Hom(B,\mathbf{Q}/\Z)\neq 0$, the case $A=\mathbf{Q}/\Z$ of Proposition \ref{cara2} characterizes locally compact groups with compact identity component that are of type $\textnormal{FP}_2$; for $A=\Z/2\Z$ it characterizes CGLC-groups with compact identity component that are of type $\textnormal{FP}_2$ mod~2.
\end{remark}

\subsection{Maximal 1-ended subgroups}\label{ss_m1e}

In a locally compact group, let us call an M1E-subgroup a compactly generated, 1-ended open subgroup that is maximal among compactly generated, 1-ended open subgroups.

\begin{lem}\label{prefa}
Let $G$ be a 1-ended CGLC-group. Then every inversion-free continuous action of $G$ on a nonempty tree with compact edge stabilizers has a unique fixed vertex.
\end{lem}
\begin{proof}
We can suppose that the action is minimal. If the tree is not reduced to a singleton, the action induces a nontrivial decomposition of $G$ as a Bass-Serre fundamental group of a graph of groups with compact open edge stabilizers, contradicting that $G$ is 1-ended. So $G$ fixes a vertex $v$; if it fixes another vertex $v'$, then it fixes every edge in between, contradicting that edge stabilizers are compact.
\end{proof}

The following lemma is straightforward.
\begin{lem}\label{m1ecq}
Let $G$ be an LC-group, and $p:G\to G/W$ the quotient by some compact normal subgroup. Then for every M1E-subgroup $H$ of $G$, $p(H)$ is an M1E-subgroup of $G/W$, and for every M1E-subgroup $L$ of $G/W$, $p^{-1}(H)$ is an M1E-subgroup of $G$.\qed
\end{lem}

\begin{lem}\label{intm1e}
Let $G$ be a compactly generated accessible LC-group and $H$ a closed cocompact subgroup. Then
\begin{enumerate}[(a)]
\item\label{gvh} for every M1E-subgroup $L$ of $G$, the intersection $L\cap H$ is an M1E-subgroup of $H$;
\item\label{hvg} every M1E-subgroup of $H$ is contained in a unique M1E subgroup of $G$.
\end{enumerate}
\end{lem}
\begin{proof}(\ref{gvh}) Define $N=L\cap H$. Since $L$ is open and $H$ is cocompact, the intersection $N$ is cocompact in $L$. So $N$ is 1-ended. By Proposition \ref{m1e}, $N$ is contained in an M1E-subgroup $P$ of $H$. Let $T$ be a tree on which $G$ acts continuously with compact edge stabilizers. Then by Lemma \ref{prefa}, each of $P$, $L$, and $N$ fix a unique vertex in $T$, and since $P\supset N\subset L$, this unique vertex $v$ is the same for all. Then $L\subset G_v\supset P$, where $G_v$ is the stabilizer of $v$. Moreover, $G_v$ is 1-ended and since $L$ is M1E, we deduce that $G_v=L$. Thus $P\subset L$. Hence $P\subset L\cap H=N$, and therefore $P=N$.

(\ref{hvg}). Let $L$ be a M1E subgroup of $H$. By Proposition \ref{m1e}, $L$ is contained in a M1E-subgroup $M$ of $G$. Since $H$ is cocompact in $G$ and $M$ is open, the intersection $H\cap M$ is cocompact in $M$; in particular, it is 1-ended and open in $H$. Since $L\subset H\cap M$ is an M1E-subgroup of $G$, we deduce that $L=H\cap M$. Hence $L$ is cocompact in $M$, showing the existence. To prove the uniqueness, let $M'$ be another M1E-subgroup of $G$ containing $L$. As in the previous paragraph, using the inclusions $M\supset L\subset M'$, using a tree with a continuous $G$-action with compact edge stabilizers and 1-ended vertex stabilizers, and applying Lemma \ref{prefa}, we obtain that all fix a common vertex $v$, and since $M\subset G_v\supset M'$ and $M,M'$ are M1E-subgroups, we obtain $M=G_v=M'$.
\end{proof}

\begin{prop}\label{m1e}
Let $G$ be an accessible LC-group and $H$ a closed, 1-ended compactly generated subgroup. Then $H$ is contained in a M1E-subgroup.
Moreover, $H$ is an M1E-subgroup if and only if there is a continuous inversion-free action of $G$ on a tree, with compact edge stabilizers, such that $H$ is the stabilizer of some vertex.
\end{prop}
\begin{proof}
Assume that $H$ of some vertex stabilizer in some inversion-free continuous action on a tree with compact edge stabilizers. Then $H$ is open. Let $L$ be a 1-ended subgroup containing $H$. By Lemma \ref{prefa}, $L$ fixes some unique vertex $v'$. By uniqueness of the vertex fixed by $H$ (again Lemma \ref{prefa}), we have $v=v'$. Since $H$ is the stabilizer of $v$, we deduce that $H=L$. So $H$ is a ME1-subgroup, proving one implication of the second statement.

Now start again with the assumptions of the proposition. Since $G$ is accessible, it admits a continuous action on a tree $T$ with no inversion, with compact edge stabilizers and (at most 1)-ended edge stabilizers. By Lemma \ref{prefa}, $H$ is contained in a vertex stabilizer $L$. Then $L$ is not compact, hence it is 1-ended. From the previous paragraph of this proof, we know that $L$ is a M1E-subgroup, which proves the first statement. If moreover $H$ is M1E, we deduce $H=L$, which establishes the converse implication of the second statement.
\end{proof}

\section{Commable groups, and commability classification of amenable hyperbolic groups}
\label{sec:com}

Let us recall some definitions and results from \cite{Co13}.

For a homomorphism between LC-groups, write {\em copci} as a (pronounceable) shorthand for {\em continuous proper with cocompact image}.

We say that a hyperbolic LC-group is {\em faithful} if it has no nontrivial compact normal subgroup. A simple observation \cite[Lemma 3.6(a)]{CCMT} shows that any hyperbolic LC-group $G$ has a maximal compact normal subgroup $W(G)$; in particular $G/W(G)$ is faithful. Note that if $G$ is faithful then $G^\circ$ is a connected Lie group.

\subsection{Generalities}
Let us say that two LC-groups $G,H$ are {\em commable} if there exist an integer $k$ and a sequence of copci homomorphisms

\begin{equation} G=G_0\moins G_1\moins G_2\moins\dots\moins G_k=H,\label{ggih}\end{equation}
where each sign $\moinss$ denotes an arrow in either direction.

We sometimes use fancy arrows such as $\nearrow$, $\nwarrow$ to make it more readable. For instance, if there are three copci homomorphisms $G\leftarrow G_1\to G_2\leftarrow H$, we can write that $G$ is commable to $H$ through $\nwarrow\nearrow\nwarrow$, or that there is a commability $G\nwarrow\nearrow\nwarrow H$.

More generally, if $\mathcal{D}$ is a class of locally compact groups and $G,H\in\mathcal{D}$, we say that $G,H$ are commable within the class $\mathcal{D}$ if the same condition holds with the additional requirement that the LC-groups $G_i$ in (\ref{ggih}) belong to $\mathcal{D}$.

We will especially consider for $\mathcal{D}$ the class of focal (resp.\ faithful, resp.\ focal and faithful) hyperbolic LC-groups.

\begin{remark}\label{bs}
To be compactly generated is invariant by commability among LC-groups. In particular, any two commable CGLC-groups are quasi-isometric. The converse is not true, even for finitely generated groups. Examples are $\Gamma_1\ast\Z$ and $\Gamma_2\ast\Z$, when $\Gamma_1,\Gamma_2$ are cocompact lattices in $\SL_3(\R)$ that are not abstractly commensurable, according to an unpublished observation of Carette and Tessera (see \S\ref{noncommable}).

Here is a possible other example, of independent interest. Let $\BS(m,n)$ be the Baumslag-Solitar group, defined by the presentation $\langle t,x\mid tx^mt^{-1}=x^n\rangle$. Then, by a result of Whyte \cite{Wh01}, the groups $\BS(2,3)$ and $\BS(3,5)$ are quasi-isometric; he also makes the simple observation that these groups do not have isomorphic finite index subgroups (i.e., since they are torsion-free, are not commable within discrete groups). I do not know whether $\BS(2,3)$ and $\BS(3,5)$ are commable, but I expect a negative answer. More generally, it would be of great interest to determine the commability and quasi-isometry classes (within locally compact groups) of any Baumslag-Solitar group $\BS(p,q)$. Their external quasi-isometry classification is addressed, within discrete groups, in \cite{MSW2}. 
\end{remark}


\begin{remark}
If $G$ and $H$ are commable, we can define their commability distance as the least $k$ such that there exists a commability such as in (\ref{ggih}). I do not know if this can be greater than 4. Anyway, this notion forgets the direction of the arrows; for instance, to have a commensuration through $\nearrow\nwarrow$ or $\nwarrow\nearrow$ are very different relations.
\end{remark}

\begin{lem}[Proposition 2.7 in \cite{Co13}]\label{fdag}
Any copci homomorphism $f:G_1\to G_2$ between LC-groups satisfies $f^{-1}(W(G_2))=W(G_1)$. In particular, if $W(G_1)$ is compact (e.g., if $G_1$ is hyperbolic), $f$ thus factors through an injective copci homomorphism 
$f':G_1/W(G_1)\to G_2/W(G_2)$, which is injective.\qed
\end{lem}

\begin{cor}\label{cfcckp}
Let $\mathcal{F}$ denote the class of focal hyperbolic LC-groups. Given faithful hyperbolic LC-groups $G_1$ and $G_2$,
\begin{itemize}
\item $G_1$ and $G_2$ are commable if and only they are commable within faithful hyperbolic groups;
\item assuming $G_1$ and $G_2$ focal, $G_1$ and $G_2$ are commable within focal hyperbolic groups if and only they are commable within faithful focal hyperbolic groups.\qed
\end{itemize}
\end{cor}

\begin{remark}
We can define {\em strict commability} in the same fashion as commability, but only allowing {\em injective} copci homomorphisms. Since a copci homomorphism between faithful focal hyperbolic LC-groups is necessarily injective, Corollary \ref{cfcckp} shows that between faithful hyperbolic LC-groups, commability and strict commability are the same, and more generally that hyperbolic LC-groups $G_1,G_2$ are commable if and only if $G_1/W(G_1)$ and $G_2/W(G_2)$ are strictly commable. Similar consequences hold for commability within focal hyperbolic LC-groups.
\end{remark}

\subsection{Non-commable quasi-isometric discrete groups}\label{noncommable}

Let us abbreviate totally disconnected to TD; thus compact-by-TD LC-groups are the same as Hausdorff topological groups with a compact open subgroup. Also, recall that maximal 1-ended (M1E) subgroups were introduced in \S\ref{ss_m1e}.

\begin{lem}\label{com1e}
Let $G,H$ be compactly generated, accessible locally compact groups with infinitely many ends. Assume that $G$ and $H$ are commable. Then every M1E-subgroup of $G$ is commable within compact-by-TD subgroups to a M1E subgroup of $H$.
\end{lem}
\begin{proof}Since $G$ and $H$ have infinitely many ends, they are compact-by-TD as well as all LC-groups commable to them, and their closed subgroups. The result then follows, by an immediate induction, from Lemmas \ref{m1ecq} and \ref{intm1e}.
\end{proof}


\begin{lem}\label{comadis}
Let $\Gamma$ be a finitely generated group with no nontrivial finite normal subgroup. Assume that every compact-by-TD LC-group quasi-isometric to $\Gamma$ is compact-by-discrete and has a maximal compact normal subgroup.

Then a compact-by-TD LC-group $H$ is commable to $\Gamma$ within compact-by-TD LC-groups if and only if its compact normal subgroup $W$ exists and $H/W$ is commensurable (= strictly commable within discrete groups) to $\Gamma$.
\end{lem}
\begin{proof}
Consider a sequence of copci homomorphisms $\Gamma=G_0\moins G_1\moins\dots\moins G_k=H$ between compact-by-TD LC-groups. When $G$ is an LC-group, let $\mathsf{W}(G)$ denote the union of all compact normal subgroups; when $G$ admits a maximal compact normal subgroup, it is equal to $\mathsf{W}(G)$. Note that $\mathsf{W}(\Gamma)=\{1\}$ by the first assumption. Then any copci homomorphism $U\to V$ maps $\mathsf{W}(U)$ into $\mathsf{W}(V)$ (this follows from Proposition \ref{fdag}). It follows that the above sequence induces a sequence of homomorphisms $\Gamma\moins G_1/\mathsf{W}(G_1)\moins\dots\moins G_k/\mathsf{W}(G_k)=H/\mathsf{W}(H)$, which are copci homomorphisms since $\mathsf{W}(G_i)$ is compact for all $i$, by the assumptions. All these groups being discrete, this means that $\Gamma$ and $H/\mathsf{W}(H)$ are commensurable.
\end{proof}

\begin{example}[Carette-Tessera]
Let $S$ be a noncompact connected semisimple Lie group with trivial center. Let $\Gamma,\Lambda$ be non-commensurable cocompact lattices in $S$ (this exists unless $S$ is isomorphic to the product of a compact group with $\SL_2(\R)$). Then $\Gamma\ast\Z$ and $\Lambda\ast\Z$ are quasi-isometric but not commable, as we now show.

A consequence of Theorem \ref{cqi} is that every LC-group quasi-isometric to $S$ is compact-by-Lie. In particular, every compact-by-TD LC-group quasi-isometric to $S$ is compact-by-discrete. Accordingly, the assumption of Lemma \ref{comadis} is satisfied by $\Gamma$, and it follows this lemma that $\Gamma$ and $\Lambda$ are not commable within compact-by-TD LC-groups. 

Since $S$ is 1-ended, so are $\Gamma$ and $\Lambda$, and therefore these are M1E-subgroups in $\Gamma\ast\Z$ and $\Lambda\ast\Z$. Thus, by Lemma \ref{com1e}, we deduce that $\Gamma\ast\Z$ and $\Lambda\ast\Z$ are not commable.

To show that $\Gamma\ast\Z$ and $\Lambda\ast\Z$ are quasi-isometric, we first need to know that $\Gamma$ and $\Lambda$ are bilipschitz (the result follows by an immediate argument). That they are bilipschitz follows from a result of Whyte \cite{Wh99} asserting that any two quasi-isometric non-amenable finitely generated groups are bilipschitz.
\end{example}

The next three subsections address commability within focal groups. See \S\ref{speci} for the link with commability.

\subsection{Connected type}\label{sct}

We say that a focal hyperbolic group is {\em focal-universal} if it satisfies the following: for every LC-group $H$, the group $H$ is commable to $G$ within focal groups if and only if there is a copci homomorphism $H\to G$.

\begin{thm}[\cite{Co13}]\label{efo}
Let $G$ be a focal hyperbolic LC-group of connected type. Then $G$ is commable to a focal-universal LC-group $\hat{G}$, which thus satisfies: for every LC-group $H$, the group $H$ is commable to $G$ within focal groups if and only if there is a copci homomorphism $H\to \hat{G}$. Moreover, $\hat{G}$ is unique up to topological isomorphism.
\end{thm}

\begin{cor}\label{ckco}
Any two commable focal hyperbolic groups of connected type are commable within focal groups through $\nearrow\nwarrow$.\qed
\end{cor}




Focal-universal groups of connected type give a canonical set of representatives of commability classes of focal hyperbolic LC-groups of connected type. However, another canonical set of representatives, sometimes more convenient (e.g.\ in view of the classification small dimension) is given by purely real Heintze groups.

\begin{prop}\label{rhei}
Every focal-universal LC-group of connected type has a unique closed cocompact subgroup that is purely real Heintze. Every focal hyperbolic group of LC-type is commable to a purely real Heintze group, unique up to isomorphism.\qed
\end{prop}

\begin{example}
Here is (without proof) the classification of purely real Heintze groups in dimension $2\le d\le 4$. For each isomorphy class we give one representative.
\begin{itemize}
\item $d=2$: the only group is the affine group $\Carn(\R)=\R\rtimes\R$, where the action is given by $t\cdot x=e^tx$. It is isomorphic to a closed cocompact subgroup in $\Isom(\mathbf{H}^2_\R)$.
\item $d=3$: the groups are the $G_\lambda$ for $\lambda\ge 1$, as well as a certain $G_1^{u}$. All these groups are defined as a semidirect product $\R^2\rtimes\R$. In $G_\lambda$, the action is given by $t\cdot (x,y)=(e^tx,e^{\lambda t}y)$, while in $G_1^u$ it is given by $t\cdot(x,y)=e^t(x+ty,y)$. Note that $G_1=\Carn(\R^2)$ is isomorphic to a closed cocompact subgroup of $\Isom(\mathbf{H}^3_\R)$. 
\item $d=4$. It consists of semidirect products $\R^3\rtimes\R$ and $\mathsf{Hei}_3\rtimes\R$, where $\mathsf{Hei}_3$ is the 3-dimensional Heisenberg group. More precisely, the semidirect products $\R^3\rtimes\R$ consist of the $G_{\lambda,\mu}$ for $1\le \lambda\le\mu$, with action given by $t\cdot (x,y,z)=(e^tx,e^{\lambda t}y,e^{\mu t}z)$, of the $G^{(u)}_\lambda$ ($\lambda>0$) for which the action is given by $t\cdot (x,y,z)=(e^t(x+ty),e^ty,e^{\lambda t}z)$, and of $G^u_1$ for which the action is given by $t\cdot (x,y,z)=e^t(x+ty+t^2z/2,y+tz,z)$. If we use coordinates for the Heisenberg group so that the product is given by $(x,y,z)(x',y',z')=(x+x',y+y',z+z'+xy'-x'y)$ (these are not the standard matrix coordinates!), the corresponding Heintze groups are $H_\lambda$ for $\lambda\ge 1$ with action given by $t\cdot (x,y,z)=(e^tx,e^{\lambda t}y,e^{(1+\lambda)t}z)$, and $H^u_1$ with action given by $t\cdot (x,y,z)=e^t(x+y,y,z)$. Note that $G_{1,1}=\Carn(\R^3)$ is isomorphic to a closed cocompact subgroup of $\Isom(\mathbf{H}^4_\R)$, and $H_1=\Carn(\mathsf{Hei}_3)$ is isomorphic to a closed cocompact subgroup of $\Isom(\mathbf{H}^2_\C)$.
\item For $d=5$, we do not give the full classification, but just mention that the groups occurring are semidirect products $\R^4\rtimes\R$, $(\mathsf{Hei}_3\times\R)\rtimes R$, or $\mathsf{Fil}_4\rtimes\R$, for various contracting actions we do not describe, where $\mathsf{Fil}_4$ is the filiform 4-dimensional Lie group, whose Lie algebra has a basis $(a,e_1,e_2,e_3)$ so that $[a,e_i]=e_{i+1}$ for $i=1,2$ and all other brackets between basis elements vanish. The Carnot groups $\Carn(\mathsf{Hei}_3\times\R)$ and $\Carn(\mathsf{Fil}_4)$ are the smallest examples of Carnot groups that are not isomorphic to closed cocompact groups of isometries of rank 1 symmetric spaces of noncompact type.
\item For $d'\le 6$, there are finitely many isomorphism classes of simply connected Carnot-gradable nilpotent Lie groups of dimension $d'$ (namely 1,1,2,3,7,21 isomorphism classes for $d'=1,\dots,6$), while for $d'\ge 7$ there are continuously many. Hence for $d\le 7$ there are finitely many isomorphism classes of purely real Heintze groups of Carnot type, and continuously many for $d\ge 8$.
\end{itemize}
\end{example}



\subsection{Totally disconnected type}\label{tdty}

Define, for every integer $m\ge 2$, $\FT_m$ as the stabilizer of a given boundary point in the automorphism group of an $(m+1)$-regular tree.

If $G$ is any locally compact group having exactly two (opposite) continuous homomorphisms onto $\Z$ (e.g., any focal hyperbolic group not of connected type), for $n\ge 1$ we denote by $G^{[n]}$ the inverse image of $n\Z$ by any of these homomorphisms.


By {\em non-power} integer we mean an integer $q\ge 1$ that is not an integral proper power of any integer (thus excluding $4,8,9,16,25,27,32,36,49\dots$).



An easy observation (Proposition \ref{tdck}) is that any two focal hyperbolic LC-groups $G_1,G_2$ of totally disconnected type are always commable. However, they are not always commable within focal groups:

\begin{prop}[\cite{Co13}]\label{tdckfo}
Let $G_1,G_2$ be focal hyperbolic LC-groups of totally disconnected type. The following are equivalent:
\begin{enumerate}
\item\label{fcck1} $G_1$ and $G_2$ are commable within focal groups;
\item\label{fcck4} $\Delta(G_1)$ and $\Delta(G_2)$ are commensurable subgroups of $\R_+$ (i.e.\ $\Delta(G_1)\cap\Delta(G_2)$ has finite index in both).
\item\label{fcck5} There exists a non-power integer $q\ge 2$ such that $\Delta(G_1)$ and $\Delta(G_2)$ are both contained in the multiplicative group $\langle q\rangle=\{q^n:n\in\Z\}$;
\item\label{fcck2} there is a commability within focal groups $G_1\nearrow\nwarrow\nearrow\nwarrow G_2$;
\item\label{fcck8} there exists a non-power integer $q\ge 2$ and an integer $n\ge 1$, such that for $i=1,2$ there is a there is a commability within focal groups $G_i\nearrow\nwarrow \FT_{q}^{[n]}$;
\item\label{fcck6} there exists a non-power integer $q\ge 2$ such that for $i=1,2$, there is a commability within focal groups with $G_i\nearrow\nwarrow\nearrow \FT_q$;
\item\label{fcck3} there is a commability within focal groups $G_1\nwarrow\nearrow\nwarrow\nearrow G_2$;
\item\label{fcck7} there exists an integer $m\ge 2$, such that for $i=1,2$ there is a there is a commability within focal groups $G_i\nwarrow\nearrow \FT_m$.
\end{enumerate}
\end{prop}

\subsection{Mixed type}




If $G$ is a locally compact group, its elliptic radical $G^\sharp$ is its largest closed elliptic normal subgroup, where elliptic means that every compact subset is contained in a compact subgroup. If $G$ is focal of mixed type, then $G/G^\sharp$ is focal of connected type.

\begin{defn}[\cite{Co13}]\label{dvarpi}Let $G$ be a focal hyperbolic LC-group. Consider the modular functions of $G/G^\circ$ and $G/G^\sharp$; by composition they define homomorphisms $\Delta_G^{\textnormal{td}}$, $\Delta_G^{\textnormal{con}}:G\to\R_+$, which we call restricted modular functions.
Since $\Hom(G,\R)$ is 1-dimensional, if $G$ is not of totally disconnected type then $\Delta_G^{\textnormal{con}}$ is nontrivial and hence there exists a unique $\varpi=\varpi(G)\in\R$ such that $\log\circ\Delta_G^{\textnormal{td}}=\varpi(\log\circ\Delta_G^{\textnormal{con}})$. Because of the compacting element in $G$, necessarily $\varpi(G)\ge 0$, with equality if and only if $G$ is of connected type. If $G$ is of totally disconnected type we set $\varpi(G)=+\infty$.
\end{defn}

\begin{defn}\label{dhpa}Let $H,A$ be a focal hyperbolic LC-groups, $H$ being of connected type with a surjective modular function and $A$ being of totally disconnected type. For $\varpi>0$, define
\[H\stackrel{\varpi}{\times}A=\{(x,y)\in H\times A\mid \Delta_H(x)^\varpi=\Delta_{A}(y)\}.\]
This is a focal hyperbolic LC-group of mixed type, satisfying $\varpi(H\stackrel{\varpi}{\times}A)=\varpi$. If $q\ge 2$ is an integer and $\varpi$ is a positive real number, define in particular
\[H[\varpi,q]=H\stackrel{\varpi}{\times}\FT_q.\]
\end{defn}

\begin{prop}[\cite{Co13}]\label{micom}
Let $G_1,G_2$ be focal hyperbolic LC-groups of mixed type. Equivalences:
\begin{enumerate}
\item\label{mi1} $G_1$ and $G_2$ are commable;
\item\label{gdgsv} the following three properties hold:
\begin{itemize}
\item $G_1/G_1^\circ$ and $G_2/G_2^\circ$ are commable within focal groups;
\item 
$G_1/G_1^\sharp$ and $G_2/G_2^\sharp$ are commable within focal groups\item $\varpi(G_1)=\varpi(G_2)$;\end{itemize}
\item\label{mi3} there is a commability within focal groups $G_1\nearrow\nwarrow\nearrow\nwarrow G_2$;
\item\label{mi5} there exists a non-power integer $q\ge 2$, an integer $n\ge 1$, a focal-universal group of connected type $H$ and a positive real number $\varpi>0$ such that for $i=1,2$ there is a there is a commability within focal groups $G_i\nearrow\nwarrow H[q,\varpi]^{[n]}$;
\item\label{mi6} there exists a non-power integer $q\ge 2$, a focal-universal group of connected type $H$ and a positive real number $\varpi>0$ such that for $i=1,2$, there is a commability within focal groups with $G_i\nearrow\nwarrow\nearrow H[\varpi,q]$;
\item\label{mi4} there is a commability within focal groups $G_1\nwarrow\nearrow\nwarrow\nearrow G_2$;
\item\label{mi7} there exists an integer $m\ge 2$, a focal-universal group of connected type $H$ and a positive real number $\varpi>0$ such that for $i=1,2$ there is a there is a commability within focal groups $G_i\nwarrow\nearrow H[\varpi,m]$.
\end{enumerate}
\end{prop}

In view of Proposition \ref{rhei}, we deduce:

\begin{cor}\label{c_comix}
Every focal hyperbolic LC-group of mixed type $G$ is commable to an LC-group of the form $H[\varpi,q]$, for some purely real Heintze group $H$ uniquely defined up to isomorphism, a unique $\varpi\in\mathopen]0,\infty\mathclose[$, and a unique non-power integer $q$.\qed
\end{cor}

\subsection{Commability between focal and general type groups}\label{speci}\label{sspecial}


The following lemma from \cite{Co13} is essentially contained in \cite{CCMT}.

\begin{lem}\label{g12f}
Let $G_1,G_2$ be focal hyperbolic LC-groups. Equivalences:
\begin{enumerate}
\item\label{h12} there exists a hyperbolic LC-group $G$ with copci homomorphisms $G_1\to G\leftarrow G_2$. 
\item\label{hf12} there exists a focal hyperbolic LC-group $G$ with copci homomorphisms $G_1\to G\leftarrow G_2$. \qed
\end{enumerate}
\end{lem}

\begin{prop}\label{CCKFCCK}
Let $G$ be a focal hyperbolic LC-group not of totally disconnected type. Let $H$ be a locally compact group commable to $G$. Then the following statements hold.
\begin{enumerate}[(a)]
\item\label{kc1} If $H$ is focal, then $G$ and $H$ are commable within focal groups;
\item\label{kc2} if $H$ is non-focal, then there exists a rank~1 symmetric space of noncompact type $X$ and continuous proper compact isometric actions of $G$ and $H$ on $X$.
\end{enumerate}
\end{prop}
\begin{proof}
If $G$ is quasi-isometric to any rank 1 symmetric space $X$ of noncompact type, then by Theorem \ref{cqi} we have copci homomorphisms $G\to\Isom(X)\leftarrow H$; if $H$ is of general type this proves (\ref{kc2}); if $H$ is focal then conjugating by some isometry we can suppose it has the same fixed point in $\partial X$ as $G$, proving (\ref{kc1}).

Assume otherwise that $G$ is not quasi-isometric to any rank 1 symmetric space of noncompact type. To show the result, it is enough to check that the commability class of $G$ consists of focal groups. Otherwise, there is a copci homomorphism $G_1\to G_2$ between groups in the commability class of $G$, such that $G_1$ is focal and $G_2$ is not focal. By Theorem \ref{ccmti}(\ref{c4}), it follows that $G$ is quasi-isometric to some rank 1 symmetric space of noncompact type (which is excluded) or to a tree (which is excluded since $\partial G$ is positive-dimensional).
\end{proof}

\begin{prop}[\cite{Co13}]\label{tdck}
Two focal hyperbolic LC-groups $G_1,G_2$ of totally disconnected type are always commable through $\nearrow\nwarrow\nearrow\nwarrow$, and are commable to a finitely generated free group of rank $\ge 2$ through $\nearrow\nwarrow$.
\end{prop}
\begin{proof}
There is a copci homomorphism $G_i\to\Aut(T_i)$ for some regular tree $T_i$ of finite valency at least 3. For $d$ large enough, $\Aut(T_i)$ ($i=1,2$) contains a cocompact lattice isomorphic to a free group of rank $k=1+d!$. Thus we have copci homomorphisms $G_1\to\Aut(T_1)\leftarrow F_k\to \Aut(T_2)\leftarrow G_2$.
\end{proof}


\vspace{0.2cm}
\begin{center}{\bf Special groups}\end{center}
\vspace{-0.2cm}
\begin{defn}\label{dspe}
By {\em special} hyperbolic LC-group we mean any CGLC-group commable to both focal and general type hyperbolic LC-groups. A CGLC-group quasi-isometric to a special hyperbolic group is called {\em quasi-special}.
\end{defn}

See Remark \ref{termi_special} for some motivation on the choice of terminology. Taking Theorems \ref{cqi} and \ref{cqia} for granted, we have the following characterizations of special and quasi-special hyperbolic groups.

\begin{prop}\label{carspe}
Let $G$ be a CGLC-group. Equivalences:
\begin{enumerate}
\item\label{sp4} $G$ is special hyperbolic;
\item\label{sp2} $G$ is commable to the isometry group of a metric space $X$ which is either a rank~1 symmetric space of noncompact type or to a 3-regular tree.
\end{enumerate}

We also have the equivalences:
\begin{enumerate}
\addtocounter{enumi}{2}
\item\label{qs} $G$ is quasi-special hyperbolic;
\item\label{sp1} $G$ is quasi-isometric to either a rank~1 symmetric space of noncompact type or to a 3-regular tree;
\item\label{sp3} $G$ admits a continuous proper cocompact isometric action on either a rank~1 symmetric space or nondegenerate tree (see \S\ref{qista}) (which can be chosen to be regular if $G$ is focal).
\end{enumerate}
Moreover, if $G$ is quasi-special and not special, then it is of general type and quasi-isometric to a 3-regular tree.
\end{prop}

 
\begin{proof}
Let us begin by proving the equivalence (\ref{sp4})$\Leftrightarrow$(\ref{sp2}), not relying on Theorems \ref{cqi} and \ref{cqia}.

\noindent (\ref{sp2})$\Rightarrow$(\ref{sp4}) is clear since, denoting by $X$ the space in (\ref{sp2}), it follows that $G$ is commable to both $\Isom(X)$ and $\Isom(X)_\omega$ for some boundary point.

\noindent (\ref{sp4})$\Rightarrow$(\ref{sp2}) We can suppose that $G$ is focal. If $G$ is of totally disconnected type, then by the easy \cite[Proposition 4.6]{CoTe} (which makes $G$ act on its Bass-Serre tree), $G$ has a continuous proper cocompact isometric action on a regular tree and thus (\ref{sp2}) holds. Assume that $G$ is not of totally disconnected type. Let $G=H_0-H_1-H_2\dots-H_k$ be a sequence of copci arrows (in either direction) with $H_k=\Isom(X)$, which is of general type. Let $i$ be minimal such that $H_i$ is of general type. Then $i\ge 1$, $H_{i-1}$ is focal and necessarily the copci arrow is in the direction $H_{i-1}\to H_i$. So $H_i$ is hyperbolic of general type and its identity component is not compact. Thus by \cite[Proposition 5.10]{CCMT}, $H_i$ is isomorphic to an open subgroup of finite index subgroup in the isometry group of a rank~1 symmetric space of noncompact type, proving (\ref{sp2}).

Let us now prove the second set of equivalences.

\noindent (\ref{sp3})$\Rightarrow$(\ref{sp1}) is immediate, in view of Lemma \ref{tqit}.


\noindent (\ref{sp1})$\Rightarrow$(\ref{qs}) let $X$ be the space as in (\ref{sp1}); then $\Isom(X)$ is special since it is of general type and the stabilizer of a boundary point is focal and cocompact. 

\noindent (\ref{qs})$\Rightarrow$(\ref{sp3}): if $G$ is quasi-special, then it is quasi-isometric to a space as in (\ref{sp2}), so (\ref{sp3}) is provided by Theorems \ref{cqi} and \ref{cqia}, except the regularity of the tree in the focal case, in which case we can then invoke the easy \cite[Proposition 4.6]{CoTe} (which makes $G$ act on its Bass-Serre tree). 

For the last statement, observe that (\ref{sp3})$\Rightarrow$(\ref{sp2}) holds in case in (\ref{sp3}) we have a symmetric space or a regular tree.
\end{proof}

\begin{remark}In the class of hyperbolic LC-groups quasi-isometric to a non-degenerate tree (which is a single quasi-isometry class, described in Corollary \ref{cqiab}), there is a ``large" commability class, including
\begin{enumerate}
\item
all focal hyperbolic groups of totally disconnected type (by Proposition \ref{tdck});
\item
all discrete groups (i.e., non-elementary virtually free finitely generated groups);
\item more generally, all unimodular groups quasi-isometric to a non-degenerate tree (because they have a cocompact lattice by \cite{BK}), including all automorphism groups of regular and biregular trees of finite valency;
\item non-ascending HNN extensions of compact groups over open subgroups: indeed the Bass-Serre tree is then a regular tree. Such groups are non-focal are often non-unimodular.
\end{enumerate}
\end{remark}

It is natural to ask whether this ``large" class is the whole class:

\begin{question}\label{sqs}
Are any two hyperbolic LC-groups quasi-isometric to the 3-regular tree commable? Equivalently, does quasi-special imply special?
\end{question}

It is easy to check that any such group is commable to the Bass-Serre fundamental group of a connected finite graph of groups in which all vertex and edge groups are isomorphic to $\widehat{\Z}$ (the profinite completion of $\mathbf{Z}$). I expect a positive answer to Question \ref{sqs}.

{\bf Update.} This question was settled positively in full generality by M.\ Carette \cite{Ca} after being asked in a previous version of this survey and in \cite{Co13}.


A thorough study of cocompact isometry groups of bounded valency trees is carried out in \cite{MSW02}.

\begin{remark}\label{termi_special}
The adjective ``special" indicates that special hyperbolic LC-groups are quite exceptional among hyperbolic LC-groups, although they are the best-known. For instance, there are countably many quasi-isometry classes of such groups, namely one for trees and one for each homothety class of rank~1 symmetric space of noncompact type (and finitely many classes for each fixed asymptotic dimension), while there are continuum many pairwise non-quasi-isometric 3-dimensional Heintze groups.
\end{remark}

Let us mention the following lemma, which is well-known but often referred to without proof.

\begin{lem}\label{tqit}
Let $T$ be a bounded valency nonempty tree with no vertex of degree 1 and in which the set of vertices of valency $\ge 3$ is cobounded. Then $T$ is quasi-isometric to the 3-regular tree. In particular, if $T$ is a bounded valency tree with a cocompact isometry group and at least 3 boundary points then it is quasi-isometric to the 3-regular tree $T'$.
\end{lem}
\begin{proof}
A first step is to get rid of valency 2 vertices. Indeed, since there are no valency 1 vertex and by the coboundedness assumption, every valency 2 vertex $v$ lies in a unique (up to orientation) segment consisting of vertices $v_0,\dots,v_n$ with $v_0$, $v_n$ of valency $\ge 3$, and each $v_1,\dots,v_{n-1}$ being of valency 2, and $n$ being bounded independently of $v$. If we remove the vertices $v_1,\dots,v_{n-1}$ and join $v_0$ and $v_n$ with an edge, the resulting tree is clearly quasi-isometric to $T$. Hence in the sequel, we assume that $T$ has no vertex of valency $\le 2$.


Let $s\ge 3$ be the maximal valency of $T$. Denote by $T^0$ and $T^1$ the 0-skeleton and 1-skeleton of $T$. Let $T_3$ be a 3-regular tree, and fix an edge, called root edge, in both $T$ and $T_3$, so that the $n$-ball $T(n)$ means the $n$-ball around the root edge (for $n=0$ this is reduced to the edge). Let us define a map $f:T\to T_3$, by defining it by induction on the $n$-ball $T(n)$.

We prescribe $f$ to map the root edge to the root edge, and, for $n\ge 0$, we assume by induction that $f$ is defined on the $n$-ball of $T^0$. The convex hull of $f(T(n))$ is a finite subtree $A(n)$ of $T_3$. We assume the following property $\mathcal{P}(n)$: the subtree $A(n)\subset T_3^0$ has no vertex of valency 2, and the $n$-sphere of $T^0$ is mapped to the boundary of this subtree $A(n)$. Let $v$ belong to the $n$-sphere of $T$. Then $v$ has a single neighbor not in the $(n+1)$-sphere, and has a number $m\in [2,s-1]$ of other neighbors, in the $(n+1)$-sphere of $T$. The vertex $f(v)$ has a single neighbor $w_0$ in $A(n)$. (Choosing a root edge instead of a root vertex is only a trick to avoid a special step when $n=0$.)


We use the following claim: in the binary rooted tree of height $k\ge 1$ and root $o$, for every integer in $m\in [2,2^k]$ there exists a finite subset $F$ of vertices of cardinal $m$ with $o\notin F$ such that the convex hull of $F\cup\{o\}$ has exactly $F$ as set of vertices of valency 1 and admits only $o$ as set of vertices of valency 2. The proof is immediate: reducing the value of $k$ if necessary, we can suppose $m\ge 2^{k-1}$, and then, writing $m=2^{k-1}+t=2t+(2^{k-1}-t)$, choose $2^{k-1}-t$ vertices of height $k-1$, and choose the $2t$ descendants (of height $k$) of the remaining $t$ vertices of height $k-1$, to form the subset $F$, and it fulfills the claim.

Now choose $k=\lceil \log_2(s-1)\rceil$, consider the set \[M_v=\{w\in T_3^0:d(w,w_0)-1=d(w,f(v))\le k\}\] (in other words, those vertices at distance $\le k$ from $f(v)$ not in the direction of $w_0$); this is a binary tree, rooted at $f(v)$, of height $k$. Since $m\le s-1$, and $s-1\le 2^k$, we have $m\le 2^k$ and by the claim there is a finite subset $F_v\subset M\smallsetminus\{f(v)\}$ of cardinal $m$ with the required properties. So the convex hull of $f(T(n))\cup F_v$ admits, in $M$, only elements of $F_v$ as vertices of valency 1, and no vertices of valency 1.

Noting that the $M_v$ are pairwise disjoint ($v$ ranging over the $n$-sphere of $T$), we deduce that the convex hull of $f(T(n))\cup\bigcup_vF_v$ admits no vertex of valency 2 and admits $\bigcup F_v$ as set of vertices of valency 1. Now extend $f$ to the $(n+1)$-ball by choosing, for every $v$ in the $n$-sphere of $T$, a bijection between its set of neighboring vertices in the $(n+1)$-sphere and $F_v$ (recall that they have the same cardinal by construction). Then $\mathcal{P}(n+1)$ holds by construction.

By induction, we obtain a map $f:T^0\to T_3^0$; it is injective by construction, and more precisely $d(f(v),f(v'))\ge d(v,v')$ for all $v,v'$; moreover $f$ is $k$-Lipschitz. In addition, if $A(n)$ is the convex hull of the image of $T(n)$, then an immediate induction shows that $A(n)$ contains the $n$-ball $T_3(n)$, and every point is at distance $\le k$ to the image of $f$. Thus $f$ is a quasi-isometry $T^0\to T_3^0$. 
\end{proof}

\subsection{A few counterexamples}


\begin{center}{\bf Groups of connected type with no common cocompact subgroup}\end{center}
\vspace{-0.1cm}

We have seen that any two commable focal groups of connected type are commable through $\nearrow\nwarrow$. This is not true with $\nwarrow\nearrow$. The simplest example is obtained as follows: start from  the group $G=\R\rtimes\R$ (the affine group), $u=\log\circ\Delta_G$, $G_1=u^{-1}(\Z)$ and $G_2=u^{-1}(\lambda\Z)$ where $\lambda$ is irrational. Since $\Delta(G_1)\cap\Delta(G_2)=\{1\}$, it is clear that $G_1$ and $G_1$ are not $\nwarrow\nearrow$-equivalent.

\begin{center}{\bf Focal groups not acting on the same space}\end{center}
\vspace{-0.1cm}

We saw that commable focal hyperbolic LC-groups of connected type are commable through $\nearrow\nwarrow$. This is not true in other types. We give here some examples of totally disconnected type; examples of mixed type can be derived mechanically by ``adding" a connected part.

Recall from \S\ref{tdty} that all $\FT_n$, for $n\ge 2$, are commable, and that $\FT_m$ and $\FT_n$ are commable within focal groups if and only if $m$ and $n$ are integral powers of the same integer. In contrast, we have:

\begin{prop}[\cite{Co13}]\label{a4a8f}
If $2\le m<n$, there exist no hyperbolic LC-group $G$ with copci homomorphisms $\FT_m\to G\leftarrow \FT_n$. \qed
\end{prop}

\begin{cor}
If $m,n\ge 2$ are distinct, there is no proper metric space with continuous proper cocompact isometric actions of both $\FT_m$ and $\FT_n$.\qed
\end{cor}

\begin{center}{\bf Discrete groups not acting on the same space}\end{center}
\vspace{-0.1cm}

This subsection is a little tour beyond the focal case, dealing with discrete analogues of the examples in Proposition \ref{a4a8f}.

Let us mention a consequence of Theorem \ref{cqia}, already observed in \cite[Corollary 10]{MSW} with a slightly different point of view. Let $C_n$ be a cyclic group of order $n$. Recall that all discrete groups $C_n\ast C_m$, for $n\ge 2$ and $m\ge 2$ are quasi-isometric to the trivalent regular tree.

\begin{prop}\label{pqpq}
Let $(p_1,q_1)$ and $(p_2,q_2)$ be pairs of primes $\ge 3$. If $\{p_1,q_1\}\neq\{p_2,q_2\}$, then the groups $C_{p_1}\ast C_{q_1}$ and $C_{p_1}\ast C_{q_2}$ are not isomorphic to cocompact lattices in the same locally compact group, and thus do not act properly cocompactly on the same nonempty proper metric space.
\end{prop}

\begin{lem}\label{tprime}
Let $T$ be a tree with a cobounded action of its isometry group. Then it admits a unique minimal cobounded subtree $T'$ (we agree that $\emptyset\subset T$ is cobounded if $T$ is bounded). Moreover, $T=T'$ if and only if $T$ has no vertex of degree 1.
\end{lem}
\begin{proof}
Let $T'$ be the union of all (bi-infinite) geodesics in $T$; a straightforward argument shows that $T'$ is a subtree. Observe that any geodesic is contained in every cobounded subtree: indeed, any point of a geodesic cuts the tree into two unbounded components. It follows that $T'$ is contained in every cobounded subtree; by definition, $T'$ is $\Isom(T)$-invariant.

Let us show that $T'$ is cobounded. If $T$ is bounded then $T'=\emptyset$ and is cobounded by convention. So let us assume that $T$ is unbounded; then it is enough to show that $T'\neq\emptyset$, because then the distance to $T'$ is invariant by the isometry group, so takes a finite number of values by coboundedness. To show that $T'\neq\emptyset$, it is enough to show that $\Isom(T)$ has a hyperbolic element: otherwise the action of $\Isom(T)$ would be horocyclic and thus would preserve the horocycles with respect to some point at infinity, which would prevent coboundedness of the action of $\Isom(T)$.

The last statement is clear from the definition of $T'$.
\end{proof}

If $p,q\ge 2$ are numbers and $m\ge 1$, define a tree $T_{p,q,m}$ as follows: start from the $(p,q)$-biregular tree and replace each edge by a segment made out of $m$ consecutive edges. Note that if $p,q\ge 3$, the unordered pair $\{p,q\}$ is uniquely determined by the isomorphy type of $T_{p,q,m}$.

\begin{lem}\label{cpq}
Let $p,q$ be primes and let $C_p\ast C_q$ act minimally properly on a nonempty tree $T$ with no inversion. Then $T$ is isomorphic to $T_{p,q,m}$ for some $m\ge 1$.
\end{lem}
\begin{proof}
This gives $C_p\ast C_q$ as Bass-Serre fundamental group of a finite connected graph of groups $((\Gamma_v)_v,(\Gamma_e)_e)$ with finite vertex groups. This finite graph $X$ is a finite tree, because $\Hom(C_p\ast C_q,\Z)=0$. Let $v$ be a degree 1 vertex of this finite tree, and let $e$ be the oriented edge towards $v$. Then the embedding of $\Gamma_v$ into $\Gamma_e$ is not an isomorphism, because otherwise the action on the tree would not be minimal ($v$ corresponding to a degree one vertex in the Bass-Serre universal covering). 

Since vertex stabilizers are at most finite of prime order, this shows that such edges are labeled by the trivial group. This gives a free decomposition of the group in as many factors as degree 1 vertices in $X$. So $X$ has at most 2 degree 1 vertices and thus is a segment, and then shows that other vertices are labeled by the trivial group. If the number of vertices is $m+1$, this shows that $T$ is isomorphic to $T_{p,q,m}$.
\end{proof}

\begin{proof}[Proof of Proposition \ref{pqpq}]
Note these groups have no nontrivial compact normal subgroup, so the second statement is a consequence of the first by considering the isometry group of the space.

Suppose they are cocompact lattices in a single locally compact group $G$. Then $G$ is compactly generated and quasi-isometric to a tree, so by Theorem \ref{cqia} acts properly cocompactly, with no inversion and minimally on a finite valency tree $T$. By Lemma \ref{tprime}, the action of $C_{p_i}\ast C_{q_i}$ is minimal for $i=1,2$. By Lemma \ref{cpq}, the tree is isomorphic to $T_{p_i,q_i,m_i}$ for $i=1,2$, which implies $\{p_1,q_1\}=\{p_2,q_2\}$, a contradiction.
\end{proof}

\section{Towards a quasi-isometric classification of amenable hyperbolic groups}
\label{sec:conclu}

\subsection{The main conjecture}

We use the notion of commability studied in Section \ref{sec:com}. Here is the main conjecture about quasi-isometric rigidity of focal hyperbolic LC-groups.

\begin{conj}\label{mainc}
Let $G$ be a focal hyperbolic LC-group. Then any compactly generated locally compact group $H$ is quasi-isometric to $G$ if and only if it is commable to $G$.
\end{conj}

An LC-group $H$ as in the conjecture is necessary non-elementary hyperbolic. Thus the conjecture splits into two distinct issues:

\begin{itemize}
\item (internal case) when $H$ is focal, in which case the conjecture can be restated as: two focal hyperbolic LC-groups are quasi-isometric if and only if they are commable; this is discussed in \S\ref{internal};
\item (external case) when $H$ is of general type; this is discussed in \S\ref{qia}.
\end{itemize}

The commability classes of focal hyperbolic LC-groups having been fully described in \S\ref{sec:com} (except in the totally disconnected type), Conjecture \ref{mainin} provides a comprehensive description. Note that unlike in the two other types, in the totally disconnected type, the quasi-isometric classification is known (and trivial, since there is a single class) but the commability classification is still an open question (Question \ref{sqs}).


\subsection{The internal classification}\label{internal}

\begin{center}{\bf The main internal QI-classification conjecture}\end{center}\vspace{-0.1cm}

Let us repeat the internal part of Conjecture \ref{mainc}:

\begin{conj}\label{mainin}
Let $G$ be a focal hyperbolic LC-group. Then any focal hyperbolic LC-group $H$ is quasi-isometric to $G$ if and only if it is commable to $G$.
\end{conj}

In other words, any two focal hyperbolic LC-groups are quasi-isometric if and only if they are commable.
The conjecture is stated in a less symmetric formulation so that it makes sense to state that the conjecture holds for a given $G$. Note that any two focal hyperbolic LC-groups of totally disconnected type are commable by Proposition \ref{tdck}, so there is no need to discard them as in Conjecture \ref{mainc}.

Note that the ``if" part is trivial. Thus the conjecture is a putative description of the internal quasi-isometry classification of focal hyperbolic groups (and thus of the spaces associated to these groups), using the description of commability classes, which is described in a somewhat satisfactory way in Section \ref{sec:com}.



Conjecture \ref{mainin} can in turn be split into the connected and mixed cases.




\begin{center}{\bf Internal QI-classification in connected type}\end{center}\vspace{-0.1cm}

The following was originally stated as a theorem by Hamenst\"adt in her PhD thesis \cite{Ham}, who told me to rather consider it as a conjecture:


\begin{conj}\label{int_con}
Let $H$ be a purely real Heintze group. If a purely real Heintze group $L$ is quasi-isometric to $H$, then it is isomorphic to $H$. 
\end{conj}

In other words, the conjecture states that two purely real Heintze groups are quasi-isometric if and only if they are isomorphic. The non-symmetric formulation of the conjecture is convenient because it makes sense to assert that it holds for a given~$H$. Note that the same statement holds true for commability, by Corollary \ref{c_comix}.

Note that by Corollary \ref{ckco} and Proposition \ref{rhei}, an equivalent formulation of the conjecture consists in replacing both times ``purely real Heintze group" by ``faithful focal-universal hyperbolic LC-group of connected type". We also have

\begin{prop}\label{eqi}
Conjecture \ref{mainin} specified to groups of connected type is equivalent to Conjecture \ref{int_con}.
\end{prop}
\begin{proof}
Assume that Conjecture \ref{mainin} holds for groups of connected type. If $H_1$ is purely real Heintze and is quasi-isometric to $H$, by the validity of Conjecture \ref{mainin}, $H$ and $H_1$ are commable; hence by Proposition \ref{CCKFCCK} are commable within focal groups. By Theorem \ref{efo}, there is a faithful focal-universal LC-group $G$ and copci homomorphisms $H\to G\leftarrow H_1$. Viewing these homomorphisms as inclusions, by Proposition \ref{rhei}, we obtain that $H=H_1$. 

Conversely, suppose that Conjecture \ref{int_con} holds. Let $G_1,G_2$ be quasi-isometric focal hyperbolic groups of connected type. By Proposition \ref{rhei}, they are commable to purely real Heintze groups $H_1$ and $H_2$, which are isomorphic by Conjecture \ref{int_con}. Hence $G_1$ and $G_2$ are commable. 
\end{proof}

The results of Section \ref{sec:com}, or alternatively the more general Gordon-Wilson approach (see Remark \ref{gw}), shows the following evidence for Conjecture \ref{int_con}.

\begin{prop}\label{prh}
Two purely real Heintze groups admit continuous simply transitive isometric actions on the same Riemannian manifold if and only if they are isomorphic.
\end{prop}
\begin{proof}
If the groups are 1-dimensional, the only possibility is $\R$. Assume they have dimension $\ge 2$; then they are focal hyperbolic of connected type and commable, hence isomorphic by the results of \cite{Co13} (see Corollary \ref{c_comix}).
\end{proof}

\begin{remark}\label{gw}
The results of Gordon and Wilson \cite{GW} show that Proposition \ref{prh} holds in a much greater generality, namely for purely real simply connected solvable Lie groups (sometimes called real triangulable groups). Indeed, let $H_1,H_2$ be such groups and $X$ the Riemannian manifold. Then $H_1$ and $H_2$ stand as closed subgroups in the isometry group of the homogeneous Riemannian manifold $\Isom(X)$. Gordon and Wilson define a notion of ``subgroup in standard position" in $\Isom(X)$. In \cite[Theorem 4.3]{GW}, they show that any real triangulable subgroup is in standard position, and thus $H_1$ and $H_2$ are in standard position. In \cite[Theorem 1.11]{GW}, they show that in $\Isom(X)$, all subgroups in standard position are conjugate (and are actually equal, in case $\Isom(X)$ is amenable, see \cite[Corollary 1.12]{GW}). Thus $H_1$ and $H_2$ are isomorphic.
\end{remark}

Here are some partial results towards Conjecture \ref{int_con}.


\begin{itemize}
\item Conjecture \ref{int_con} holds for $H$ when $H$ is abelian, by results of Xie~\cite{xx}.
\item Conjecture \ref{int_con} holds for $H$ when the purely real Heintze group $H$ is cocompact in the group of isometries of a rank 1 symmetric space of noncompact type. This follows from Theorem \ref{cqi} (combined, for instance, with Proposition \ref{eqi}).
\item  If purely real Heintze groups $H_1,H_2$ of Carnot type (see \S\ref{scarnot}) are quasi-isometric,
then Pansu's Theorem \cite[Theorem~2]{pansu89m} implies that $H_1$ and $H_2$ are isomorphic.
\item Pansu's estimates of $L^p$-cohomology in degree~1 \cite{Pan07} provide useful quasi-isometry invariants.
\item Carrasco \cite[Cor.~1.9]{Carr} proves the following: given a Heintze group $H=N\rtimes\R$, let 
$\mathfrak{n}^{\min}$ be the characteristic subspace relative to the smallest eigenvalue of some dilating element of $\R$, and $H^{\min}=\exp(\mk{n}^{\min})\rtimes\R$. Then he proves that if purely real Heintze groups $H_1,H_2$ are quasi-isometric, then $H_1^{\min}$ and $H_2^{\min}$ are isomorphic. He also proves that being of Carnot type is a quasi-isometry invariant.
\end{itemize}


\begin{center}{\bf Internal QI-classification in the mixed type}\end{center}\vspace{-0.1cm}

In mixed type, we can specify Conjecture \ref{mainin} as follows (see Definitions \ref{dvarpi} and \ref{dhpa})

\begin{conj}\label{int_mi}
Let $H$ be a nonabelian purely real Heintze group, $\varpi>0$ a positive real number, and $q$ a non-power integer, and define $G=H[\varpi,q]$. If $(H',\varpi',q')$ is another such triple and $G$ and $G'=H'[\varpi',q']$ are quasi-isometric then they are isomorphic.
\end{conj}

\begin{prop}
Conjecture \ref{mainin} specified to groups of mixed type is equivalent to Conjecture \ref{int_mi}.
\end{prop}
\begin{proof}
Suppose that Conjecture \ref{mainin} specified to groups of mixed type holds. If  $G$ and $G'$ are given as in Conjecture \ref{int_mi}, then the validity of Conjecture \ref{mainin} implies that $G$ and $G'$ are commable. By Corollary \ref{c_comix}, we deduce that $G$ and $G'$ are isomorphic.

Conversely assume that Conjecture \ref{int_mi} holds. Let $G$ be as in Conjecture \ref{mainin}, of mixed type, and let $G'$ be a focal hyperbolic LC-group, quasi-isometric to $G$. Then $G'$ is necessarily of mixed type (by Corollary \ref{foc3}). By Corollary \ref{c_comix}, $G$ and $G'$ are respectively commable to groups of the form $H[\varpi,q]$ and $H'[\varpi',q']$, which by the validity of Conjecture \ref{int_mi} are isomorphic, so that $G$ and $G'$ are commable.
\end{proof}

The following theorem indicates that a significant part of Conjecture \ref{int_mi} holds, and provides a full reduction to the connected case Conjecture \ref{int_con}.

\begin{thm}\label{th_mi}
Let $G=H[\varpi,q]$ and $G'=H[\varpi,q']$, as in Conjecture \ref{int_mi}, be quasi-isometric. Then the following statements hold:
\begin{enumerate}[(a)]
\item\label{com1}\cite{Co13} $H$ and $H'$ are quasi-isometric;
\item\label{com2}\cite{Co13} $\varpi=\varpi'$;
\item\label{com3} {\bf (T.~Dymarz \cite{Dy12})} $q=q'$.
\end{enumerate}
In particular, if $H$ satisfies Conjecture \ref{int_con} then $H[\varpi,q]$ satisfies Conjecture \ref{int_mi}.\qed
\end{thm}

Theorem \ref{th_mi} shows that Conjecture \ref{int_mi} boils down to Conjecture \ref{int_con}. However, the proofs of (\ref{com1}) and especially of (\ref{com3}) in \cite{Dy12} suggest that Conjecture \ref{int_mi} might be easier than Conjecture \ref{int_con}, because its boundary exhibits more rigidity.

Let us indicate an application from \cite{Co13}. Let $X$ be a homogeneous negatively curved manifold of dimension $\ge 2$. For $t>0$, let $X_{\{t\}}$ be obtained from $X$ by multiplying the Riemannian metric by $t^{-1}$ (thus multiplying the distance by $t^{-1}$ and the sectional curvature by $t$). For instance, $\mathbf{H}^2_{\{t\}}$ is the rescaled hyperbolic plane with constant curvature $-t$.

\begin{thm}[\cite{Co13}]
If $k_1,k_2$ are integers $\ge 2$ and $t_1,t_2$ are positive real numbers, then $X_{\{t_1\}}[k_1]$ and $X_{\{t_2\}}[k_2]$ are quasi-isometric if and only if $\log(k_1)/t_1=\log(k_2)/t_2$ and $k_1,k_2$ have a common integral power.\qed

In particular, when either $k\ge 2$ or $t>0$ is fixed, the $X_{\{t\}}[k]$ are pairwise non-quasi-isometric.
\end{thm}

The proof indeed consists in proving that if $G$ is a focal hyperbolic LC-group with a continuous proper cocompact isometric action on $X_{\{t\}}[k]$, then $\varpi(G)=c\log(k)/t$, where the constant $c>0$ only depends on $X$. In particular, the last statement of the theorem follows from the quasi-isometric invariance of $\varpi$, while the first statement follows from it as well as Dymarz' invariance of the invariant $q$, and for the (easier) converse, relies on Proposition \ref{micom}.

\subsection{Quasi-isometric amenable and non-amenable hyperbolic LC-groups}\label{qia}
The external classification can be asked in two (essentially) equivalent but intuitively different ways:
\begin{itemize}
\item Which amenable hyperbolic LC-groups are QI to hyperbolic LC-groups of general type?
\item Which hyperbolic LC-groups of general type are QI to amenable hyperbolic LC-groups?
\end{itemize}


These questions are equivalent but they can be specified in different ways.
A potential full answer is given by the following conjecture:



Recall from Definition \ref{dspe} that a hyperbolic LC-group is special if and only if it is both commable to amenable and non-amenable LC-groups and quasi-special if it is quasi-isometric to a special hyperbolic group. Such groups have a very peculiar form, see Proposition \ref{carspe} for characterizations. 


\begin{conj}\label{anaqi}
Let $G$ be a hyperbolic LC-group. Then $G$ is quasi-isometric to both an amenable and a non-amenable CGLC-group if and only if $G$ is quasi-special hyperbolic.
\end{conj}



\begin{remark}The ``if" part of the conjecture is clear. Conjecture \ref{anaqi} is a coarse counterpart to \cite[Theorem D]{CCMT}, which is transcribed here as Case (\ref{c4}) of Theorem \ref{ccmti} or as the equivalence (\ref{sp4})$\Leftrightarrow$(\ref{sp2}) of Proposition \ref{carspe}.
\end{remark}

\begin{prop}
Conjecture \ref{mainc} specified to $H$ of general type is equivalent to Conjecture \ref{anaqi}.
\end{prop}
\begin{proof}
Suppose that Conjecture \ref{mainc} holds for $H$ of general type. Let $G$ be as in Conjecture \ref{anaqi}; if $G$ is quasi-isometric to the trivalent tree then it is quasi-special; otherwise $G$ is then quasi-isometric, and hence commable by the (partial) validity of Conjecture \ref{mainc}, to both focal (not of totally disconnected type) and non-focal hyperbolic LC-groups. We conclude by Proposition \ref{carspe} that $G$ is special.


Conversely assume Conjecture \ref{anaqi} holds. Suppose that $G,H$ are quasi-isometric hyperbolic LC-groups with $G$ focal not of totally disconnected type and $H$ of general type. By the validity of Conjecture \ref{anaqi}, each of $G$ and $H$ is quasi-special, and hence special by Proposition \ref{carspe}. Thus again by Proposition \ref{carspe}(\ref{sp2}), $G$ and $H$ commable to the isometry group of a rank~1 symmetric space of noncompact type; since non-homothetic rank~1 symmetric spaces of noncompact type are not quasi-isometric, we get the same group for $G$ and $H$ and thus they are commable.
\end{proof}


We now wish to specify Conjecture \ref{anaqi}.


\begin{conj}\label{nnpr}
Let $H=N\rtimes\R$ be a purely real Heintze group of dimension $\ge 2$.
Then either $H$ is special, or $H$ is not quasi-isometric to any vertex-transitive finite valency graph.
\end{conj}

\begin{remark}
The special role played by those purely real Heintze that are ``accidentally" special make the 
conjecture delicate.
\end{remark}

\begin{lem}\label{qico}
Assume that $H$ is a non-special purely real Heintze group of dimension $\ge 2$ . Then any CGLC group $G$ quasi-isometric to $H$ is either focal hyperbolic of connected type (as defined in \S\ref{fhco}), or is compact-by-(totally disconnected). If moreover $H$ satisfies Conjecture \ref{nnpr}, then $G$ is focal hyperbolic of connected type. 
\end{lem}
\begin{proof}
If $G$ is focal, its boundary is a sphere, it is of connected type.

Otherwise $G$ is of general type.
By Theorem \ref{cqi}, $H$ is not quasi-isometric to a rank 1 symmetric space of noncompact type and therefore $G$ is compact-by-(totally disconnected). In particular, $H$ is quasi-isometric to a vertex-transitive connected finite valency graph, and this is a contradiction in case Conjecture \ref{nnpr} holds.
\end{proof}

Using a number of results reviewed above, we can relate the two conjectures.


\begin{prop}\label{concon}
Conjectures \ref{anaqi} and \ref{nnpr} are equivalent.
\end{prop}
\begin{proof}
Assume Conjecture \ref{anaqi} holds. Consider $H$ as in Conjecture \ref{nnpr}, quasi-isometric to a vertex-transitive finite valency graph $X$. Then $H$ is quasi-isometric to $\Isom(X)$; if the latter is focal, being totally disconnected, it is of totally disconnected type, hence its boundary is totally disconnected, contradicting Corollary \ref{foc3}. So $\Isom(X)$ is of general type; the validity of Conjecture \ref{anaqi} then implies that $H$ is special.


Conversely assume Conjecture \ref{nnpr} holds. Let $G$ be quasi-isometric hyperbolic LC-groups $G_1,G_2$, with $G_1$ of general type and $G_2$ focal; we have to show that $G$ is special.
By Corollary \ref{milnqi}, $G_2$ is either of totally disconnected or connected type. In the first case, $G_2$ is special. Since by Proposition \ref{carspe} being special hyperbolic is a quasi-isometry among CGLC-groups, we deduce that $G$ is special. In the second case, by Proposition \ref{rhei}, $G_2$ is commable to a purely real Heintze group $H$. If by contradiction $H$ is not special, since it satisfies Conjecture \ref{nnpr}, by Lemma \ref{qico} $G_1$ is focal, a contradiction.
\end{proof}

\begin{remark}\label{focnsqi}
Another restatement of the conjectures is that the class of non-special focal hyperbolic LC-groups
is closed under quasi-isometries among CGLC-groups.
\end{remark}


Let us now give a more precise conjecture.

\begin{conj}[Pointed sphere Conjecture]\label{nnprr}
Let $H$ be a purely real Heintze group of dimension $\ge 2$. Let $\omega$ be the $H$-fixed point in $\partial H$. If $H$ is not isomorphic to the minimal parabolic subgroup in any simple Lie group of rank one, then every quasi-symmetric self-homeomorphism of $\partial H$ fixes~$\omega$.
\end{conj}

The justification of the name is that the boundary $\partial H$ naturally comes with a distinguished point; topologically this point is actually not detectable since the sphere is topologically homogeneous, but the quasi-symmetric structure ought to distinguish this point, at the notable exception of the cases for which it is known not to do so.

\begin{figure}[!t]
\includegraphics[scale=0.4]{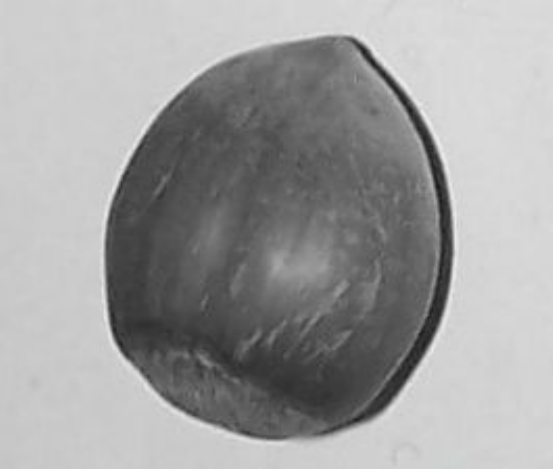}
\caption{A pointed 2-sphere: the boundary of a generic hyperbolic semidirect product $\R^2\rtimes\R$.}
\label{fig2}
\end{figure}

In turn, the pointed sphere Conjecture implies Conjecture \ref{nnpr}. More precisely:
\begin{prop}
Let $H$ satisfy the pointed sphere Conjecture. Then it also satisfies Conjecture \ref{nnpr}.
\end{prop}
\begin{proof}
Let $H$ satisfy the pointed sphere Conjecture. If $H$ is minimal parabolic, both conjectures are tautological, so assume the contrary.
Assume by contradiction that $H$ is QI to a vertex-transitive graph $X$ of finite valency. The action of $\Isom(X)$ on its boundary is conjugate by a quasi-symmetric map to a quasi-symmetric action on $\partial H$. By assumption, this action fixes a point. So $\Isom(X)$ is focal of totally disconnected type, contradicting Corollary \ref{foc3}.
\end{proof}

\begin{thm}[Pansu, Cor.~6.9 in \cite{pansu89d}]
Consider a purely real Heintze group $N\rtimes\R$, not of Carnot type (see Definition \ref{d_carnot}). Suppose in addition that the action of $\R$ on the Lie algebra $\mathfrak{n}$ is diagonalizable. Then $H$ satisfies the pointed sphere Conjecture.
\label{pansu69}\end{thm}

In case $N$ is abelian, the assumption is that the contracting action of $\R$ is not scalar (in order to exclude minimal parabolic subgroups in $\textnormal{PO}(n,1)=\Isom(\mathbf{H}^n_\R)$). 

\begin{thm}[Carrasco \cite{carr}]
The pointed sphere Conjecture for all purely real Heintze groups that are not of Carnot type.
\end{thm}

In the particular case of Heintze groups of the form $H=N\rtimes\R$ holds with $N$ abelian, this was previously proved by X.~Xie \cite{xx} with different methods.

Theorem \ref{pansu69} covers all cases when $H$ has dimension 3 (i.e.\ $N$ has dimension 2), with the exception of the semidirect product $\R^2\rtimes\R$ with action by the 1-parameter group $(U_t)_{t\in\R}$ where $U_t=\begin{pmatrix} e^t & te^t\\ 0 & e^t\end{pmatrix}$, which is dealt with specifically in \cite{xx2}. 

In Xie's Theorem (the pointed sphere conjecture for $H=N\rtimes\R$ with $N$ abelian), the most delicate case is that of an action with scalar diagonal part and nontrivial unipotent part; it is not covered by Pansu's Theorem \ref{pansu69}. In turn, the first examples of purely real Heintze groups not covered by Xie's theorem are semidirect products $\mathsf{Hei}_3\rtimes\R$, where $\mathsf{Hei}_3$ is the Heisenberg group, with the exclusion of the minimal parabolic subgroup $\Carn(\mathsf{Hei}_3)$ (see Remark \ref{grada}) in $\textnormal{PU}(2,1)=\Isom(\mathbf{H}^2_\C)$. If the action on $(\mathsf{Hei}_3)_\textnormal{ab}=\mathsf{Hei}_3/[\mathsf{Hei}_3,\mathsf{Hei}_3]$ has two distinct eigenvalues, then Pansu's Theorem \ref{pansu69} applies. The remaining case of the pointed sphere Conjecture for $\dim(H)=4$ is the one for which the action on $(\mathsf{Hei}_3)_\textnormal{ab}$ has is not diagonalizable but has scalar diagonal part, and is covered by Carrasco's theorem.


The remaining cases of the pointed sphere conjecture are now those Heintze groups of Carnot type $\Carn(N)$.
Note that if $N$ is abelian or is a generalized $(2n+1)$-dimensional Heisenberg group $\mathsf{Hei}_{2n+1}$ (characterized by the fact its 1-dimensional center equals its derived subgroup), then $\Carn(N)$ is minimal parabolic. Therefore, the first test-cases would be when $N$ is a non-abelian 4-dimensional simply connected nilpotent Lie group: there are 2 such Lie groups up to isomorphism: the direct product $\mathsf{Hei}_3\times\R$ and the filiform Lie group $\mathsf{Fil}_4$, which can be defined as $(\R[x]/x^3)\rtimes\R$ where $t\in\R$ acts by multiplication by $(1+x)^t=1+tx+\frac{t(t-1)}{2}x^2$.

\subsection{Conjecture \ref{int_con} and quasi-symmetric maps}

An important tool, given a geodesic hyperbolic space and $\omega\in\partial X$, is the visual parabolic metric on the ``parabolic boundary" $\partial X\smallsetminus\{\omega\}$. Note that unlike the visual metric, these are generally unbounded. It shares the property that any quasi-isometric embedding $f:X\to Y$ induces a quasi-symmetric embedding $\partial X\smallsetminus\{\omega\}\to\partial Y\smallsetminus\{\bar{f}(\omega)\}$. This follows from the corresponding fact for visual metrics and the fact that the embedding $\partial X\smallsetminus\{\omega\}\subset \partial X$ is quasi-symmetric \cite[Section 5]{SX}.

The parabolic boundaries and the quasi-symmetric homeomorphisms between those are therefore important tools in the study of quasi-isometry classification and notably Conjecture \ref{int_con}. Let us include the following simple lemma.

\begin{lem}
Let $G,G'$ be focal hyperbolic LC-groups and $\omega,\omega'$ the fixed points in their boundary. Suppose that there exists a quasi-isometry $f:G\to G'$. Then there exists a quasi-isometry $u:G\to G'$ such that $\bar{u}$ maps $\omega$ to $\omega'$.
\end{lem}
\begin{proof}
If $\bar{f}(\omega)=\omega'$ there is nothing to do.  Otherwise, since $G$ is transitive on $\partial G\smallsetminus\{\omega\}$, there is a left translation $v$ on $G$ such that $\bar{v}(\bar{f}^{-1}(\omega'))\neq\bar{f}^{-1}(\omega')$. Then $\bar{f}\bar{v}\bar{f}^{-1}=\overline{fvf^{-1}}$ is a quasi-symmetric self-homeomorphism of $\partial G'$ not fixing $\omega'$, so the group of quasi-isometries generated by $fvf^{-1}$ and by left translations of $G'$ is transitive on $\partial G'$. Thus after composition of $f$ by a suitable quasi-isometry in this group, we obtain a quasi-isometry $u$ such that $\bar{u}$ maps $\omega$ is mapped to $\omega'$.   
\end{proof}



If $G$ is a focal hyperbolic LC-group of connected type, with no nontrivial compact normal subgroup, and $N$ is its connected nilpotent radical, then the action of $N$ on $\partial G\smallsetminus\{\omega\}$ is simply transitive and the visual parabolic metric induces a left-invariant distance on $N$. When $G$ is of Carnot type (see \S\ref{scarnot}), this distance is equivalent (in the bilipschitz sense) to the Carnot-Caratheodory metric. In general I do not know how to describe it directly on $N$.

\subsection{Further aspects of the QI classification of hyperbolic LC-groups}This final subsection is much smaller than it should be. It turns around the general question: how is the structure of a hyperbolic LC-group of general type related to the topological structure of its boundary?

The following theorem is closely related, in the methods, to the QI classification of groups quasi-isometric to the hyperbolic plane. It should be attributed to the same authors, namely Tukia, Gabai, and Casson-Jungreis, to which we need to add Hinkkanen in the non-discrete case.

\begin{thm}
Let $G$ be a hyperbolic LC-group. Then $\partial G$ is homeomorphic to the circle if and only if $G$ has a continuous proper isometric cocompact action on the hyperbolic plane.
\end{thm}
\begin{proof}
The best-known (and hardest!) case of the theorem is when $G$ is discrete. See for instance \cite[Theorem 5.4]{kabe}. It immediately extends to the case when $G$ is compact-by-discrete (i.e.\ has an open normal compact subgroup).

Assume now that $G$ is not compact-by-discrete.
Consider the action $\alpha:G\to\Homeo(\partial G)$ on its boundary by quasi-symmetric self-homeomorphisms. Endow $\Homeo(\partial G)$ with the compact-open topology, which (for some choice of metric on $\partial G$) is the topology of uniform convergence; the function $\alpha$ is continuous. Since $G$ is non-elementary hyperbolic, the kernel of $\alpha$ is compact, and therefore by assumption the image $\alpha(G)$ is non-discrete. We can then apply Hinkkanen's theorem about non-discrete groups of quasi-symmetric self-homeomorphisms \cite{Hin}: 
after fixing a homeomorphism identifying $\partial G$ with the projective line, $\alpha(G)$ is contained in a conjugate of $\textnormal{PGL}_2(\R)$. Thus after conjugating, we can view $\alpha$ as a continuous homomorphism with compact kernel from $G$ to $\textnormal{PGL}_2(\R)$. It follows that $H=G/\Ker(\alpha)$ is a Lie group. Since $G$ is not compact-by-discrete, we deduce that the identity component $H^\circ$ is noncompact.

If $H$ is focal, then it acts continuously properly isometrically cocompactly on a millefeuille space $X[k]$; the boundary condition implies that $k=1$ (so $X[k]=X$) and $X$ is 2-dimensional, hence is the hyperbolic plane. Otherwise $H$ is of general type, and since $H^\circ$ is noncompact we deduce that $H$ is a virtually connected Lie group; being of general type it is not amenable, hence not solvable and we deduce that the image of $H$ in $\textnormal{PGL}_2(\R)$ has index at most~2. Thus in all cases $G$ admits a continuous proper cocompact isometric action on the hyperbolic plane.
\end{proof}

\begin{conj}\label{qinc}
Let $G$ be a compactly generated locally compact group. Suppose that $G$ is quasi-isometric to a negatively curved homogeneous Riemannian manifold $X$. Then $G$ is compact-by-Lie.
\end{conj}

\begin{prop}
Conjecture \ref{qinc} is implied by Conjecture \ref{anaqi}.
\end{prop}
\begin{proof}
Suppose Conjecture \ref{anaqi} holds. Let $G$ be as in Conjecture \ref{qinc}; its boundary is therefore a sphere. If $G$ is focal, then it is compact-by-Lie by Proposition \ref{foco}. Otherwise it is of general type. Since $\Isom(X)$ contains a cocompact solvable group by \cite[Proposition 1]{Hein}, we deduce from Conjecture \ref{anaqi} that $G$ has a continuous proper isometric action on a symmetric space, which implies that it is compact-by-Lie.
\end{proof}

It is a general fact that if the boundary of a non-focal hyperbolic LC-group contains an open subset homeomorphic to $\R^n$ for some $n\ge 0$, then the boundary is homeomorphic to the $n$-sphere: see \cite[Theorem 4.4]{kabe}
(which deals with the case of finitely generated groups and $n\ge 2$ but the proof works without change in this more general setting).

%
%
%
%

\begin{question}\label{bounds}
Let $G$ be a hyperbolic locally compact group whose boundary is homeomorphic to a $d$-sphere. Is $G$ necessarily compact-by-Lie?
\end{question}

A positive answer to Question \ref{bounds} would imply Conjecture \ref{qinc}, but, on the other hand would only be a very partial answer to determining which hyperbolic LC-groups admit a sphere as boundary, boiling down the question to the case of discrete groups. At this point, Question \ref{bounds} is open for all $d\ge 4$. Still, it has a positive answer for $d\le 3$ (and similarly Conjecture \ref{qinc} has a positive answer for $\dim(X)\le 4$), by the positive solution to the Hilbert-Smith conjecture in dimension $d\le 3$, which is due to Montgomery-Zippin for $d=1,2$ and J.~Pardon for $d=3$.

Interestingly, in the Losert characterization of CGLC-groups with polynomial growth \cite{Losert}, the most original part of the proof precisely consists in showing that such a group is compact-by-Lie.

\subsection{Beyond the hyperbolic case}

Here is a far reaching generalization of Conjecture \ref{int_con}. We call {\em real triangulable Lie group} a Lie group isomorphic to a closed connected subgroup of the group of upper triangular real matrices in some dimension. Equivalently, this is a simply connected Lie group whose Lie algebra is solvable, and for which for each $x$ in the Lie algebra, the adjoint operator $\mathfrak{ad}(x)$ has only real eigenvalues. The following conjecture appears in research statements I wrote around 2010 as well as earlier talks. 

\begin{conj}\label{tricon}
Two real triangulable groups $G_1,G_2$ are quasi-isometric if and only if they are isomorphic.
\end{conj}

Remark \ref{gw} provides some evidence: if $G_1,G_2$ admit isometric left-invariant Riemannian structures (or, equivalently, admit simply transitive continuous isometric actions on the same Riemannian manifold), then they are isomorphic. (This fails for more general simply connected solvable Lie groups: for instance both $\R^3$ and the universal covering of $\Isom(\R^2)$ admit such actions on the Euclidean 3-space.)

Let us mention the important result that for a real triangulable group, the dimension $\dim(G)$ is a quasi-isometry invariant, by a result of J. Roe \cite{Roe}.

A positive answer to Conjecture \ref{tricon} would entail many consequences other than the hyperbolic case, including the (internal) quasi-isometry classification of polycyclic groups. A particular case is the case of polynomial growth.

\begin{conj}\label{nilconj}
Two simply connected nilpotent Lie groups $G_1,G_2$ are quasi-isometric if and only if they are isomorphic.
\end{conj}

The latter conjecture, specified to those simply connected nilpotent Lie group admitting cocompact lattices, is equivalent to a more familiar (and complicated) conjecture concerning finitely generated nilpotent groups, namely: two finitely generated nilpotent groups are quasi-isometric if and only if they have isomorphic real Malcev closure (when a nilpotent $\Gamma$ is not torsion-free and $T$ is its torsion subgroup, its real Malcev closure is by definition the real Malcev closure of $\Gamma/T$). This is first asked by Farb and Mosher in \cite{FM00} (before their Corollary 10).

Let us mention that two finitely generated nilpotent groups are commensurable if and only they have isomorphic {\em rational} Malcev closures. Thus any pair of non-isomorphic finite-dimensional nilpotent Lie algebras $\mathfrak{g}_1,\mathfrak{g}_2$ such that $\mathfrak{g}_1\otimes\R$ and $\mathfrak{g}_2\otimes\R$ are isomorphic provides non-commensurable quasi-isometric finitely generated nilpotent groups, namely $G_1(\Z)$ and $G_2(\Z)$, where $G_i$ is the unipotent $\mathbf{Q}$-group associated to $\mathfrak{g}_i$, and some $\mathbf{Q}$-embedding $G_i\subset\GL_n$ is fixed. The smallest examples, according to the classification, are 6-dimensional. An elegant classical example is, denoting by $\mathsf{Hei}_{2n+1}$ the $(2n+1)$-dimensional Heisenberg group ($2n+1\ge 3$) and $k$ is a positive non-square integer, $\mathsf{Hei}_{2n+1}(\Z)\times \mathsf{Hei}_{2n+1}(\Z)$ and $\mathsf{Hei}_{2n+1}(\Z[\sqrt{k}])$ are not commensurable, although they are both isomorphic to cocompact lattices in $\mathsf{Hei}_{2n+1}(\R)^2$.

The main two results known in this directions are, in the setting of Conjecture \ref{nilconj}:
\begin{itemize}
\item the real Carnot Lie algebras (see Remark \ref{grada}) $\Carn(G_1)$ and $\Carn(G_2)$ are isomorphic (Pansu \cite{pansu89m});
\item (assuming that $G_1$ and $G_2$ have lattices) the real cohomology algebras of $G_1$ and $G_2$ are isomorphic (Sauer \cite{Sau06}, improving Shalom's result \cite{sh04} that the Betti numbers are equal).
\end{itemize}

Pansu's result supersedes a yet simpler result: for a simply connected nilpotent Lie group $G$, both the dimension $d$ and the degree of polynomial growth rate $\delta$ are quasi-isometry invariant. This is obvious for the growth, while for the dimension this follows from Pansu's earlier result that the asymptotic cone is homeomorphic to 
$\R^{\dim G}$ \cite{Pan83}. This covers all cases up to dimension 4, since the only possible $(d,\delta)$ are $(i,i)$ ($0\le i\le 4$), $(3,4)$, $(4,5)$, $(4,7)$, are achieved by a single simply connected nilpotent Lie group ($(i,i)$ are the abelian ones, $(3,4)$ is the 3-dimensional Heinsenberg $H_3$ and $(4,5)$ is its product with $\R$, and $(4,7)$ is the filiform 4-dimensional Lie group, of nilpotency length 3).

In dimension 5, there are 9 simply connected nilpotent Lie groups up to isomorphism. They are denoted $L_{5,i}$ in \cite{graaf} with $1\le i\le 9$, and only two are not Carnot: $L_{5,5}$ and $L_{5,6}$. Their degree of polynomial growth rate are 5, 6, 8, 6, 8, 11, 11, 7, 10 respectively. This only leaves three pairs not determined by the quasi-isometry invariance of $(d,\delta)$; the first is $L_{5,2}$ and $L_{5,4}$, where $L_{5,2}\simeq H_3\times\R^2$ and $L_{5,4}\simeq H_5$ are distinguished by Pansu's theorem. The other two pairs we discuss below are $L_{5,3}$ and $L_{5,5}$ on the one hand, $L_{5,6}$ and $L_{5,7}$ on the other hand. 


\begin{itemize}
\item For $L_{5,3}$ and $L_{5,5}$, the Betti numbers are $(1,3,4,4,3,1)$, and $\Carn(L_{5,5})\simeq L_{5,3}$ (they have growth exponent 8); thus they are distinguished neither by Pansu's theorem, nor by Shalom's theorem. A basis for the Lie algebra $\mathfrak{l}_{5,i}$ for $i=3,5$ is given as $(a,b,c,d,e)$ with for both, nonzero brackets $[a,b]=c$, $[a,c]=e$, and for $\mathfrak{l}_{5,5}$ the additional nonzero bracket $[b,d]=e$. Note that $\mk{l}_{5,3}$ is the product of the 1-dimensional Lie algebra and the 4-dimensional filiform Lie algebra.

 However, they are distinguished by Sauer's theorem: indeed, the cup product $S^2(H^2(\mathfrak{l}_{5,i}))\to H^4(\mathfrak{l}_{5,i})$ has rank 1 (in the sense of linear algebra) for $\mk{l}_{5,3}$ and 2 for $\mk{l}_{5,5}$. Thus the cohomology algebras are not isomorphic as graded algebras.

\item For both $L_{5,7}$ and $L_{5,6}$, the Betti numbers are $(1,2,3,3,2,1)$. We have $\Carn(L_{5,6})\simeq L_{5,7}$ (they have growth exponent 11); thus they are distinguished neither by Pansu's theorem, nor by Shalom's theorem.
 A basis for the Lie algebra $\mathfrak{l}_{5,i}$ for $i=7,6$ is given as $(a,b,c,d,e)$ with for both, nonzero brackets $[a,b]=c$, $[a,c]=d$, $[a,d]=e$, and for $\mathfrak{l}_{5,6}$ the additional nonzero bracket $[b,c]=e$. The Lie algebra $\mathfrak{l}_{5,7}$ is the standard filiform 5-dimensional Lie algebra.

Actually the graded cohomology algebras are isomorphic: both have a basis $(1,a_1,a_2,\penalty-10000 b_5,b_6,b_7,c_8 ,c_9,c_{10},d_{13},d_{14},e_{15})$ so that the product is commutative with unit 1, $H^1$ has basis $(a_1,a_2)$, $H_2$ has basis $(b_6,b_7,b_8)$, etc, and the nonzero products of basis elements except the unit are
\begin{itemize}
\item $a_id_{15-i}=b_{j}c_{15-j}=e_{15}$;
\item $a_1b_7=-2c_8$, $a_2b_6=c_8$, $a_2b_7=c_9$;
\item $b_6b_7=d_{13}$, $b_7b_7=-2d_{14}$.
\end{itemize}
(it is commutative because the product of any two elements of odd degree is zero). Thus $L_{5,7}$ and $L_{5,6}$, although non-isomorphic, are not distinguishable by either Pansu or Sauer's theorem. Thus they seem to be the smallest open case of Conjecture \ref{nilconj} (note that being rational, they have lattices; the rational structure being unique up to isomorphism, these lattices are uniquely defined up to abstract commensurability). Still, these can be distinguished by the adjoint cohomology: $H^1(\mathfrak{g},\mathfrak{g})$ (the space of derivations modulo inner derivations) has dimension 5 for $\mathfrak{l}_{5,7}$ and 4 for $\mathfrak{l}_{5,6}$; however it is not known if this dimension is a quasi-isometry invariant for simply connected nilpotent Lie groups. 
\end{itemize}

It would be interesting to show that the adjoint cohomology of the Lie algebra is a quasi-isometry invariant of simply connected nilpotent Lie groups (at least as a real graded vector space): indeed, as mentioned by Magnin who did comprehensive dimension computations in dimension $\le 7$ \cite{Mag08}, it is much more effective than the trivial cohomology to distinguish non-isomorphic nilpotent Lie algebras.

Indeed, the classification of 6-dimensional nilpotent Lie algebras provides 34 isomorphism classes over the field of real numbers, 26 of which are not 2-nilpotent. Among those, 13 are Carnot (over the reals); among those Lie algebras with the same Carnot Lie algebra, all have the same Betti numbers (i.e.\ the same trivial cohomology as a graded vector space) with 2 exceptions, which provide, by the way, the smallest examples for which Shalom's result improves Pansu's (Shalom \cite[\S 4.1]{sh04} provides a 7-dimensional example, attributed to Y.~Benoist). 

For the interested reader, we list the 6-dimensional real nilpotent Lie algebras, referring to \cite{graaf} for definitions, according to their Carnot Lie algebra.  Since the Betti numbers and adjoint cohomology computations are done in \cite{Mag08} with a different nomenclature, we give in each case both notations.

\begin{itemize}
\item Nilpotency length 5:
\begin{itemize}
\item Carnot: $\mathfrak{l}_{6,18}$ (standard filiform). Is the Carnot Lie algebra of 3 Lie algebras: $\mathfrak{l}_{6,i}$ for $i=18,17,15$. In \cite{Mag08}, they are denoted $\mathfrak{g}_{6,j}$ with $j=16,17,19$ (in the same order). All have the Betti numbers $(1,2,3,4,3,2,1)$. However they can be distinguished by adjoint cohomology in degree 1 (of dimension 6, 5, 4 respectively). 
\item Carnot: $\mathfrak{l}_{6,16}$. Is the Carnot Lie algebra of 2 Lie algebras: $\mathfrak{l}_{6,i}$ for $i=16,14$. Both have the Betti numbers $(1,2,2,2,2,2,1)$.
\end{itemize}
\item Nilpotency length 4:
	\begin{itemize}
	\item Carnot: $\mathfrak{l}_{6,7}$ (product of a 5-dimensional standard filiform with an 1-dimensional abelian one). Is the Carnot Lie algebra of 5 Lie algebras: $\mathfrak{l}_{6,i}$ for $i=7$, 6, 11, 12, 13. In \cite{Mag08}, they are denoted $\mathfrak{g}_{5,5}\times\R$, $\mathfrak{g}_{5,6}\times\R$, and $\mathfrak{g}_{6,j}$ for $j=12,11,13$. The first four have the Betti numbers $(1,3,5,6,5,3,1)$, while the last one has the Betti numbers $(1,3,4,4,4,3,1)$ and thus the corresponding group can be distinguished by Shalom's theorem (incidentally, the first 4 are metabelian while the last one is not). 
	\item Each of the last three Carnot Lie algebras of nilpotency length 4 are only Carnot Lie algebras of itself. They are denoted $\mathfrak{l}_{6,21}(\eps)$ for $\eps=0,1,-1$; in \cite{Mag08} they are denoted $\mathfrak{g}_{6,14}$ and $\mathfrak{g}_{6,15}$ (twice, the last two having isomorphic complexifications).
	\end{itemize}
\item Nilpotency length 3:
	\begin{itemize}
	\item Carnot: $\mathfrak{l}_{6,9}$. It is the Carnot Lie algebra of 4 Lie algebras (3 over the complex numbers): $\mathfrak{l}_{6,9}$, $\mathfrak{l}_{6,24}(1)$, $\mathfrak{l}_{6,24}(-1)$, and $\mathfrak{l}_{6,24}(0)$. (In \cite{Mag08}, these are $\mathfrak{g}_{5,4}\times\R$, $\mathfrak{g}_{6,5}$ (twice), and $\mathfrak{g}_{6,8}$.) All have the same Betti numbers $(1,3,5,6,5,3,1)$. The dimension of the zeroth and first adjoint cohomology distinguishes them, except the two middle ones having isomorphic complexification.  
	\item Carnot: $\mathfrak{l}_{6,25}$. It is the Carnot Lie algebra of 2 Lie algebras: $\mathfrak{l}_{6,i}$ for $i=25,23$. (In \cite{Mag08}, these are $\mathfrak{g}_{6,j}$ for $j=6,7$.) They both have the Betti numbers $(1,3,6,8,6,3,1)$. The dimension of the first adjoint cohomology distinguished them.
	\item Carnot: $\mathfrak{l}_{6,3}$. It is the Carnot Lie algebra of 3 Lie algebras: $\mathfrak{l}_{6,i}$ for $i=3,5,10$. (In \cite{Mag08}, these are $\mathfrak{g}_{4}\times\R^2$, $\mathfrak{g}_{5,3}\times\R$, and $\mathfrak{g}_{6,2}$.) The first two have the Betti numbers $(1,4,7,8,7,4,1)$, while the last one has the Betti numbers $(1,4,6,6,6,4,1)$. The first two, still, can be distinguished by adjoint cohomology in degree 0 (i.e., they have centers of distinct dimension), and also in degree 1.
\item Each of the last four Carnot Lie algebras of nilpotency length 3 are only Carnot Lie algebras of itself. They are denoted $\mathfrak{l}_{6,19}(0)$, $\mathfrak{l}_{6,19}(1)$, $\mathfrak{l}_{6,19}(-1)$, and $\mathfrak{l}_{6,20}$, and are called in \cite{Mag08} $\mathfrak{g}_{6,4}$, $\mathfrak{g}_{6,9}$ for the two middle one which have isomorphic complexification, and $\mathfrak{g}_{6,10}$. They have Betti numbers $(1,3,6,8,6,3,1)$ for the first one and $(1,3,5,6,5,3,1)$ for the last three ones.
		\end{itemize}
\item Nilpotency length $\le 2$. They are all Carnot and thus distinguished by Pansu's theorem. They are denoted $\mathfrak{l}_{6,26}$, $\mathfrak{l}_{6,22}(\eps)$ for $\eps=0,1,-1$, and $\mathfrak{l}_{6,i}$ for $i=8,4,2,1$, and in \cite{Mag08} they are denoted $\mathfrak{g}_{6,3}$, $\mathfrak{g}_{6,1}$, $\mathfrak{g}_{3}\times\mathfrak{g}_{3}$ (twice), $\mathfrak{g}_{5,2}\times\R$, $\mathfrak{g}_{5,1}\times\R$, $\mathfrak{g}_{3}\times\R^3$, and $\R^6$.
\end{itemize}

Note that in all the cases above distinguished neither by Pansu nor by Shalom's theorem, I have not computed the cup product in cohomology and thus have not checked in which cases they are distinguished by Sauer's theorem.

\end{document}